\documentclass[12pt]{article}

\usepackage{amssymb}
\usepackage{graphics}

\newcommand{\qed}{$\;\;\;\Box$}
\newenvironment{proof}{\par\smallbreak{\sl\bf Proof.~}}
{\unskip\nobreak\hfill \qed \par\medbreak}

\newcounter{claim}
\renewcommand{\theclaim}{\arabic{claim}}
{\par\medskip\par}

{\qed\par\smallbreak}
\newcommand{\D}{{\cal D}}

\newcommand{\N}{{\mathbb N}}
\newcommand{\R}{{\mathbb R}}

\newcommand{\Z}{{\mathbb Z}}

\newcommand{\CC}{{\cal C}}

\newcommand{\G}{{\cal G}}

\newcommand{\LL}{{\cal L}}
\newcommand{\F}{{\cal F}}
\newcommand{\HH}{{\cal H}}

\newcommand{\beq}{\begin{equation}}
\newcommand{\ee}{\end{equation}}

\renewcommand{\d}{\partial}

\newtheorem{thm}{Theorem}[section]
\newtheorem{lem}[thm]{Lemma}

\newtheorem{cor}[thm]{Corollary}
\newtheorem{rem}[thm]{Remark}

\newcommand{\de}{\delta}
\newcommand{\eps}{\varepsilon}

\newcommand{\la}{\lambda}
\newcommand{\om}{\tau}

\newcommand{\reff}[1]{(\ref{#1})}      

\title{Solution regularity and smooth dependence \\ 
for abstract equations and applications to \\ hyperbolic PDEs
} 

\newcounter{thesame}
\author{
I.~Kmit
 \ \ \ L.~Recke\\
{\small
Institute of Mathematics, Humboldt University of Berlin,}
\\
{\small Rudower Chaussee 25, D-12489 Berlin, Germany }
\\
{\small
and Institute for Applied Problems of Mechanics and Mathematics, }
\\
{\small
Ukrainian Academy of Sciences,  Naukova St.\ 3b, 79060 Lviv,
Ukraine 
}
\\
{\small   E-mail:
{\tt kmit@informatik.hu-berlin.de}}\\[5mm]
{\small
Institute of Mathematics, Humboldt University of Berlin,}\\
{\small 
Rudower Chaussee 25, D-12489 Berlin, Germany}\\
{\small   E-mail:
{\tt recke@mathematik.hu-berlin.de}}
}
 
\date{}

\begin{document}

\maketitle

\begin{abstract}
\noindent
In the first part we present a generalized implicit function theorem for abstract equations of the type $F(\la,u)=0$.
We suppose that $u_0$ is a solution for $\la=0$ and that $F(\la,\cdot)$ is smooth for all $\la$, but, mainly, we do not suppose that 
$F(\cdot,u)$ is smooth for all $u$. Even so, we state conditions such that for all $\la \approx 0$ there 
exists exactly one solution $u \approx u_0$,  that $u$ is smooth in a certain abstract sense,
and that the data-to-solution map $\la \mapsto u$ is smooth.

In the second part we apply the  results of the first part
to time-periodic solutions of first-order hyperbolic systems of the type
$$
\partial_tu_j  + a_j(x,\la)\partial_xu_j + b_j(t,x,\la,u)  = 0, \; x\in(0,1), \;j=1,\dots,n
$$
with reflection boundary conditions and of second-order hyperbolic equations of the type
$$
\d_t^2u-a(x,\la)^2\d^2_xu+b(t,x,\la,u,\d_tu,\d_xu)=0, \; x\in(0,1)
$$
with mixed boundary conditions (one Dirichlet and one Neumann).
There are at least two  distinguishing features of these results in comparison with the corresponding ones for parabolic PDEs:
First, one has to prevent small divisors from coming up, and we present  explicit sufficient conditions for that
in terms of $u_0$ and of the data of the PDEs and of the boundary conditions.
And second, in general smooth dependence of the coefficient functions $b_j$ and  $b$ on $t$ is needed in order to get 
smooth dependence of the solution on $\la$,
this is completely different to what is known for parabolic PDEs.
\end{abstract}


\section{Introduction}\label{sec:intr}
\renewcommand{\theequation}{{\thesection}.\arabic{equation}}
\setcounter{equation}{0}

In the first part of this paper we consider abstract parameter depending equations of the type
\beq
\label{abst}
F(\la,u)=0.
\ee
We suppose $F(0,u_0)=0$
and state conditions on $F$ and $u_0$ such 
that for all $\la \approx 0$ there exists exactly one solution $u \approx u_0$ to \reff{abst}
and that the data-to-solution map $\la \mapsto u$ is smooth.
Two conditions for that are, similar to classical implicit function theorems, that
 $F(\la,\cdot)$ is smooth for all $\la \approx 0$ and  that $\d_uF(0,u_0)$ is an isomorphism.
But, mainly, we do not suppose that  $F(\cdot,u)$ is smooth for all $u \approx u_0$.
In our applications to hyperbolic PDEs the map $(\la,u) \mapsto\d_uF(\la,u)$ is even not continuous with respect 
to the uniform operator norm, in general.
In other words: We consider parameter depending equations, 
which do not depend smoothly on the parameter, but with solutions which do  depend smoothly on the parameter.
For that, of course, some additional structure is needed, which will be described in Section \ref{IFT}.

Moreover, we prove an abstract solution regularity result for the equation \reff{abst}
of the following kind: Let 
$$
\ldots \hookrightarrow U_{l+1} \hookrightarrow U_{l} \hookrightarrow \ldots  \hookrightarrow U_1  \hookrightarrow U_0
$$
be a sequence
of Banach spaces continuously embedded into each other. Suppose that $F$ maps $\R \times U_0$ into $U_0$
and satisfies some weak smoothness condition with respect to $\la$ (see \reff{ksmooth}) and 
some strong smoothness condition with respect to $u$ (see \reff{infsmooth}).
Then for all solutions $(\la,u) \in \R \times U_0$ to
\reff{abst}, which are close to $(0,u_0)$, it holds $u \in U_l$ for all $l \in \N$, and the data-to-solution map 
$\la \mapsto u\in U_l$ is smooth for all  $l\in\N$.

In the second part we apply the abstract results of the first part to  semilinear first-order
hyperbolic systems of the type 
\beq\label{eq:1.1}
\partial_tu_j  + a_j(x,\la)\partial_xu_j + b_j(t,x,\la,u)  = 0, \; x\in(0,1), \;j=1,\dots,n,
\ee
with periodicity conditions in time
\beq\label{eq:1.3}
u_j(t+2\pi,x) = u_j(t,x), \; x\in(0,1),\;j=1,\dots,n,
\ee
and reflection boundary conditions in space
\beq\label{eq:1.2}
\begin{array}{l}
\displaystyle
u_j(t,0) = \sum\limits_{k=m+1}^nr_{jk}(t,\la)u_k(t,0), \;  j=1,\ldots,m,\\
\displaystyle
u_j(t,1) = \sum\limits_{k=1}^mr_{jk}(t,\la)u_k(t,1),  \;  j=m+1,\ldots,n,
\end{array}
\ee
(or, more generally, with the boundary conditions  \reff{eq:1.2b}) 
as well as to semilinear second-order hyperbolic equations of the type
\beq\label{eq:1.1a}
\partial_t^2u-a(x,\la)^2\d^2_xu+b(t,x,\la,u,\d_tu,\d_xu)=0, \; x\in(0,1) 
\ee
with periodicity condition in time
\beq\label{eq:1.3a}
u(t+2\pi,x) = u(t,x), \; x\in(0,1)
\ee
and  boundary conditions in space
\beq\label{eq:1.2a}
u(t,0) = \d_xu(t,1)=0.
\ee

We suppose that for $\la=0$ there exists a classical solution $u_0$ to 
(\ref{eq:1.1})--(\ref{eq:1.2}) (respectively, to (\ref{eq:1.1a})--(\ref{eq:1.2a})). We state conditions on $u_0$ and on the data of (\ref{eq:1.1})--(\ref{eq:1.2})
(respectively, of (\ref{eq:1.1a})--(\ref{eq:1.2a}))
such that for all $\la \approx 0$ there exists exactly one solution $u \approx u_0$  to 
(\ref{eq:1.1})--(\ref{eq:1.2})  (respectively, to (\ref{eq:1.1a})--(\ref{eq:1.2a})), that this solution is smooth (with respect to $t$ and $x$)
and that the data-to-solution map $\la \mapsto u$ is smooth (in any $C^k$-norm). 

We do not assume that the time-periodic solution $u_0$ is close to be time-independent.
Moreover, we do not assume that $u_0$ is $C^\infty$-smooth with respect to $t$ and $x$, but we prove that it is  $C^\infty$-smooth (under reasonable assumptions, cf. Theorems \ref{thm:hopf} 
and \ref{thm:hopfa}).
Remark also that the initial-boundary value problems, corresponding to (\ref{eq:1.1}), (\ref{eq:1.2}) 
(respectively,  to (\ref{eq:1.1a}),(\ref{eq:1.2a})), do not have a smoothing property 
(for $t>t_0$ the solution $u(t,\cdot)$
is not smoother than the initial function  $u(t_0,\cdot)$), in general.
In fact, all our results concerning the periodic-boundary value problems  (\ref{eq:1.1})--(\ref{eq:1.2}) and (\ref{eq:1.1a})--(\ref{eq:1.2a}) will be proven 
without using any properties of the corresponding initial-boundary value problems.

For smoothness of the data-to-solution map, for example, with respect to  the $L^\infty$-norm,
we need not only  smoothness of the coefficient functions with respect to $\la$, $u$, $\d_tu$ and $\d_xu$, but also  smoothness 
with respect to $t$ (cf. Remark \ref{mn1}).
This is completely different to what is known for elliptic and parabolic PDEs,
where the data-to-solution map may be smooth  with respect to  the $L^\infty$-norm even if the coefficient functions are discontinuous 
in time and/or space variables
(see \cite{GriRe,GroRe}).

Also, we have to assume the conditions \reff{Fred} (or \reff{Fred1} in the case $m=1, n=2$) for (\ref{eq:1.1})--(\ref{eq:1.2}) and  \reff{Ra} for (\ref{eq:1.1a})--(\ref{eq:1.2a}) 
which do not have an analog in corresponding parabolic problems. For example, for  (\ref{eq:1.1})--(\ref{eq:1.2}) with $m=1$, $n=2$, 
$r_{jk}(t,0)=r_{jk}$ and
$b_j(t,x,0,u)=b_{j1}(x)u_1+b_{j2}(x)u_1+f_j(t,x)$
the condition reads
$$
|r_{12}r_{21}|\exp\int_0^1\left(\frac{b_{11}(x)}{a_1(x,0)}-\frac{b_{22}(x)}{a_2(x,0)}\right)dx
\not=1.
$$
In \cite{KR1} it is shown that  this is a kind of a nonresonance condition which prevents small divisors from coming 
up in the Fourier series for the solution components $u_j$.
For  (\ref{eq:1.1a})--(\ref{eq:1.2a}) with $b(t,x,0,u,\d_tu,\d_xu)=a_1(x)\d_tu+a_2(x)\d_xu+b_0(t,x,u)$ the 
nonresonance condition reads
$$
\int_0^1\frac{a_1(x)}{a(x,0)}\,dx\not=0.
$$

\begin{rem}\rm
\label{general}
Unfortunately we do not know if generalizations of our results to  cases with   higher space dimensions and/or to
quasilinear equations exist and how they should look like. 
On the other hand,  generalization of  our results to  cases of a multidimensional control parameter $\la$ 
is straightforward (cf. \cite{KR5} and  Remark \ref{moredim}).
\end{rem}

\begin{rem}\rm
\label{Ausn}
(i) If the coefficients $a_j$ in (\ref{eq:1.1})--(\ref{eq:1.2})
(respectively,  $a$ in  (\ref{eq:1.1a})--(\ref{eq:1.2a})) depend on $t$, then the question
 of smoothness of the data-to-solution map 
seems to be much more difficult (cf. Remark~\ref{tabh}).
This is again quite different to what is known for parabolic PDEs.

(ii) In the particular case, if the coefficients $a_j$  (resp. $a$) do not depend on $\la$ and if the coefficients $b_j$  (resp. $b$)
are not  independent on $t$, the proof of the smoothness of the
data-to-solution map is much simpler and can be done by means of the classical implicit function theorem, cf. Remarks~\ref{nonsm}
and \ref{auto}.
\end{rem}

Our paper is organized as follows:

In Section \ref{IFT} we formulate and prove a generalized implicit function theorem.
Here we use, among other technical tools, the  converse to Taylor's theorem 
and  the fiber contraction principle. 
Section  \ref{Proof} presents our results concerning regularity and
smooth dependence of solutions to (\ref{eq:1.1})--(\ref{eq:1.2}),
using integration along characteristics 
as well as the results of  Section~\ref{IFT}.
Similar results for the second-order hyperbolic problem (\ref{eq:1.1a})--(\ref{eq:1.2a})
are obtained in  Section~\ref{wave}.
Finally, in  \ref{AppendixA}  we present a simple linear version of the fiber contraction principle, while in 
 \ref{AppendixB}
we develop an abstract approach how to verify the key assumptions \reff{FH} and \reff{coerz} of the generalized 
implicit function theorem.

\section{A generalized implicit function theorem}
\label{IFT}
\renewcommand{\theequation}{{\thesection}.\arabic{equation}}
\setcounter{equation}{0}

\subsection{Setting and main result}
\label{setting1}

In this section we consider abstract parameter depending equations of the type \reff{abst}.
Suppose  that
\beq
\label{zero}
F(0,u_0)=0.
\ee
We are going to state conditions on $F$ and $u_0$ such 
that for all $\la \approx 0$ there exists exactly one solution $u \approx u_0$ to \reff{abst},
that the solutions $u$  are smooth (in a way defined below) and smoothly depend  on $\la$.
One condition for that is that $F(\la,\cdot)$ is smooth and that  $\d_uF(\la,u_0)$ is an isomorphism
for all $\la \approx 0$.
But, mainly, we do not suppose that  $F(\cdot,u)$ is smooth for all $u \approx u_0$.
In particular, we do not suppose that  $\d_uF$ is continuous, hence, in general, the condition $\d_uF(\la,u_0)\in
\mbox{Iso}$ for all $\la \approx 0$ does not follow from the  condition $\d_uF(0,u_0) \in
\mbox{Iso}$ of the  classical implicit function theorem (see Remark \ref{rema1}).

In other words, we are going to describe parameter depending equations, which do not depend smoothly on the parameter,
but with solutions which do  depend smoothly on the parameter.
For that, of course, some additional structure is needed, which will be described now.

Let $U_0$ be a Banach space with norm $\|\cdot\|_0$, and let $T(s) \in \LL(U_0), s \in \R$, be a strongly continuous group of 
linear bounded operators on $U_0$. Denote by $A: D(A) \subseteq U_0 \to U_0$ the infinitesimal generator of $T(s)$ and by
$$
U_l:=D(A^l)=\{u \in U_0: T(\cdot)u \in C^l(\R;U_0)\}
$$
the domain of definition of the $l$-th power of $A$. Because of $A$ is closed, $U_l$ is a Banach space with the norm
$$
\|u\|_l:=\sum_{k=0}^l\|A^ku\|_0.
$$
Further, let $u_0 \in U_0$ and $\eps_0>0$ be given. Let  $F:[-\eps_0,\eps_0] \times U_0 \to U_0$ 
be a map satisfying  \reff{zero}.
Consider the map  $\F:[-\eps_0,\eps_0]\times \R  \times U_0 \to U_0$  defined by 
\beq
\label{FFdef} 
\F(\la,s,u):=T(s)F(\la,T(-s)u). 
\ee
Suppose that 
\beq
\label{infsmooth}
\F(\la,\cdot,\cdot) \in C^\infty(\R \times U_0;U_0) \mbox{ for all } \la \in  [-\eps_0,\eps_0],
\ee 
and denote by $\d_s^k\d_u^j\F(\la,s,u)$ 
the corresponding partial derivatives, i.e.  bounded multilinear maps from $U_0^j:=U_0\times \ldots \times U_0$ ($j$ factors) into $U_0$.
Further, we suppose that for all nonnegative integers $j,k$ and $l$
\beq
\label{ksmooth}
\d_s^k\d_u^j\F(\cdot,0,u)(u_1,\ldots,u_j) \in C^l([-\eps_0,\eps_0]; U_0) \mbox{ for all } u,u_1,\ldots,u_j \in U_l.
\ee
Finally, for $u,u_1,\ldots,u_j \in U_l$ denote by $\d_\la^l\d_s^k\d_u^j\F(\la,s,u)(u_1,\ldots,u_j) \in U_0$ the  derivative of order $l$ of the map $\d_s^k\d_u^j\F(\cdot,0,u)(u_1,\ldots,u_j)$.
Suppose that for all nonnegative integers $j,k,l$ 
there exists $c_{jkl}>0$ such that for all $u, u_1,\ldots,u_j \in U_l$ with $ \|u-u_0\|_l \le 1$ and for all  $\la \in [-\eps_0,\eps_0]$ it holds
\beq
\label{apriori}
\|\d_\la^l\d_s^k\d_u^j\F(\la,0,u)(u_1,\ldots,u_j)\|_0 \le  c_{jkl}
\|u_1\|_l \ldots \|u_j\|_l.
\ee

\begin{rem}\rm
\label{care}
The notation $\d_\la^l\d_s^k\d_u^j\F(\la,s,u)(u_1,\ldots,u_j)$ 
should be used with some care: By definition,
\beq
\label{Abl}
\d_\la^l\d_s^k\d_u^j\F(\la,s,u)(u_1,\ldots,u_j):=\frac{d^l}{d\mu^l}\left[\d_s^k\d_u^j\F(\la+\mu,s,u)(u_1,\ldots,u_j)\right]_{\mu=0}
\ee
for given $\la \in  [-\eps_0,\eps_0], s \in \R$ and $u,u_1,\ldots,u_j\in U_l$, where the derivatives in the right-hand side of \reff{Abl} 
exist in the sense of the norm $\|\cdot\|_0$ in $U_0$ because of assumption \reff{ksmooth}.
From the other side, assumption \reff{apriori} claims that for given $\la \in  [-\eps_0,\eps_0], s \in \R$ and $u\in U_l$
the symmetric multilinear map  $(u_1,\ldots,u_j)\in U_l^j \mapsto \d_\la^l\d_s^k\d_u^j\F(\la,s,u)(u_1,\ldots,u_j) \in U_0$
is bounded, hence it is reasonable to denote it by
$$
\d_\la^l\d_s^k\d_u^j\F(\la,s,u):= \d_\la^l\d_s^k\d_u^j\F(\la,s,u)(\cdot,\ldots,\cdot) \in {\cal S}_j\left(U_l;U_0\right).
$$
 Here and in what follows we denote by ${\cal S}_j\left(U_l;U_0\right)$ the vector space of all  
bounded symmetric multilinear maps
from $U_l^j$ to $U_0$ with the usual uniform operator norm
$$
\|A\|_{{\cal S}_j\left(U_l;U_0\right)}:=\sup\left\{\|A(u_1,\ldots,u_j)\|_0:\; \|u_1\|_l \le 1,\ldots,\|u_j\|_l \le 1\right\}.
$$
For example, in general $\d_\la\d_u\F(\la,s,u)$ is not the limit with respect to the uniform operator norm in $\LL(U_0)$
of the differential quotient $(\d_u\F(\la+\mu,s,u)-\d_u\F(\la,s,u))/\mu$ for $\mu \to 0$.

We proceed similarly with $F(\la,u)=\F(\la,0,u)$:  Given $\la \in  [-\eps_0,\eps_0]$ and $u,u_1,\ldots,u_j\in U_l$,
denote 
\beq
\label{derdef}
\d_\la^l\d_u^jF(\la,u)(u_1,\ldots,u_j):=\frac{d^l}{d\mu^l}\left[\d_u^jF(\la+\mu,u)(u_1,\ldots,u_j)\right]_{\mu=0}.
\ee
Because of  \reff{apriori} we have
$$
\d_\la^l\d_u^jF(\la,u):= \d_\la^l\d_u^jF(\la,u)(\cdot,\ldots,\cdot) \in {\cal S}_j\left(U_l;U_0\right).
$$
We do not suppose that $F(\cdot,u): [-\eps_0,\eps_0] \to U_0$ is differentiable for all $u \in U_0$ or that $\d_uF: [-\eps_0,\eps_0]\times U_0 \to \LL(U_0)$
is continuous (with respect to the uniform operator norm in $\LL(U_0)$),
and in the hyperbolic problems discussed in Sections \ref{Proof}  and \ref{wave} this is indeed not  the case.
But later on (see Lemma \ref{lemln3}) we will show that under the assumptions \reff{infsmooth}-\reff{apriori} 
the map $F$ is $C^m$-smooth from $[-\eps_0,\eps_0]\times U_{l+m}$ to $U_l$  for all  $l$ and $m$.
\end{rem}

The main result of this section is the following theorem:
\begin{thm}
\label{thm:IFT}
Suppose \reff{zero} and \reff{infsmooth}--\reff{apriori}. Assume that there exists $c_0>0$ such that 
for all  $\la \in [-\eps_0,\eps_0]$
\beq
\label{FH}
\d_uF(\la,u_0) \mbox{ is Fredholm of index zero from $U_0$ to $U_0$}
\ee
and 
\beq
\label{coerz}
\|\d_uF(\la,u_0)u\|_0 \ge c_0\|u\|_0 \mbox{ for all } u \in U_0.
\ee
Then there exist $\eps \in (0,\eps_0]$ and $\delta>0$ such that for all  $\la \in [-\eps,\eps]$ there exists
a unique  solution $u=\hat{u}(\la)$ to \reff{abst} with $\|u-u_0\|_0 \le \delta$. Moreover, for all nonnegative integers
$k$
we have $\hat{u}(\la) \in U_k$, and the map $\la \in  [-\eps,\eps] \mapsto \hat{u}(\la) \in U_k$ is $C^\infty$-smooth.
\end{thm}

\begin{rem}\rm
\label{hardIFT}
The assumptions \reff{FH} and \reff{coerz} imply that the operators $\d_uF(\la,u_0)$ are isomorphisms on $U_0$.
Hence, 
the difference between  Theorem \ref{thm:IFT}
and the classical  implicit function theorem is not a degeneracy of the partial derivatives  $\d_uF(\la,u)$
(like in  implicit function theorems of Nash-Moser type), but a  degeneracy of the partial derivatives  $\d_\la F(\la,u)$
(which do not exist for all $u \in U_0$, in general). 
There exist several generalizations of the  classical implicit function theorem in which  the partial derivatives  $\d_\la F(\la,u)$
do not exist  for all $u$ or  in which  the partial derivatives  $\d_u F(\la,u)$
are not continuous with respect to $(\la,u)$,
see, e.g.,
\cite[Theorem 7]{Appell}, 
 \cite[Theorem 3.4]{Fife}, \cite[Theorem 4.1]{FifeG},  \cite[Theorems 2.1 and 2.2]{Renardy1}. 
However, it turns out that they do not fit to our applications, so we are going to adapt ideas of  \cite{KR4},
\cite{Magnus}
and \cite[Theorem 2.1]{RO}.
\end{rem}

\begin{rem}\rm
\label{moredim}
We assume that the control parameter $\la$ is one-dimensional, what makes the presentation simpler. But Theorem \ref{thm:IFT} can be generalized in a straightforward way
to the case of maps $F: \Lambda_0 \times U_0\mapsto U_0$ where $\Lambda_0$ is a neighborhood of zero in a normed vector space $\Lambda$. In this case $\d_\la^l\d_s^k\d_u^j\F(\la,s,u)$
is not a multilinear map from $U_0^j$ to $U_0$ anymore, but  a multilinear map from $\Lambda^l \times U_0^j$ to $U_0$, and the assumption \reff{apriori} has to be changed to
$$
\|\d_\la^l\d_s^k\d_u^j\F(\la,0,u)(\la_1,\ldots,\la_l,u_1,\ldots,u_j)\|_0 \le  c_{jkl}
\|\la_1\|\ldots\|\la_l\|\|u_1\|_l \ldots \|u_j\|_l,
$$
where $\|\cdot\|$ is the norm in $\Lambda$.
\end{rem}

\subsection{Proof of Theorem \ref{thm:IFT}}
\label{Proof2.1}

In this subsection we prove Theorem \ref{thm:IFT}. Hence we suppose all assumptions of this theorem to be fulfilled.

First we show local existence and uniqueness for the equation \reff{abst} by means of  the Banach fixed point theorem.
The only difference between our proof and the proof of  the classical implicit function theorem is that we consider the fixed point problem \reff{Gdef}
below, while in the proof of the classical implicit function theorem it suffices to consider the slightly simpler  fixed point problem
$u-\d_uF(0,u_0)^{-1}F(\la,u)=u.$

\begin{lem}
\label{exun}
There exist $\eps \in (0,\eps_0]$ and $\delta>0$ such that for all  $\la \in [-\eps,\eps]$ there exists
a unique solution $u=\hat{u}(\la)$ to \reff{abst} with $\|u-u_0\|_0 \le \delta$.
\end{lem}
\begin{proof}
Take  $\la \in [-\eps_0,\eps_0]$.  Because of the assumptions \reff{FH} and  \reff{coerz}, the 
operator $\d_uF(\la,u_0)$ is bijective from $U_0$ to $U_0$.
Hence, the equation \reff{abst} is equivalent to the fixed point problem
\beq
\label{Gdef}
G(\la,u):=u-\d_uF(\la,u_0)^{-1}F(\la,u)=u.
\ee
Moreover, we have
$$
G(\la,u_1)-G(\la,u_2)=\int_0^1\left(I-\d_uF(\la,u_0)^{-1}\d_uF(\la,tu_1+(1-t)u_2)\right)(u_1-u_2)\,dt
$$
and 
$$
\left(I-\d_uF(\la,u_0)^{-1}\d_uF(\la,u)\right)v=-\d_uF(\la,u_0)^{-1}\int_0^1\d_u^2F(\la,tu_0+(1-t)u)(u_0-u,v)\,dt.
$$
Hence, assumptions \reff{apriori} and \reff{coerz} yield that for all positive $\delta<1$ 
$$
\|\left(I-\d_uF(\la,u_0)^{-1}\d_uF(\la,u)\right)v\|_0 \le \frac{\delta c_{200}}{c_0}\|v\|_0 \mbox{ for all } u \in B_\delta(u_0), v \in U_0,
$$
where $ B_\delta(u_0):=\{u \in U_0: \|u-u_0\|_0 \le \delta\}$ is the closed ball of radius $\delta$ around $u_0$.
Therefore, if $\delta$ is sufficiently small,  
then
\beq
\label{est}
\|I-\d_uF(\la,u_0)^{-1}\d_uF(\la,u)\|_{\LL(U_0)} \le \frac{1}{2} \mbox{ for all } u \in B_\delta(u_0),
\ee
and hence
\beq
\label{Gest}
\|G(\la,u_1)-G(\la,u_2)\|_0 \le \frac{1}{2}\|u_1-u_2\|_0  \mbox{ for all } u_1,u_2 \in B_\delta(u_0).
\ee
Further, for $ u\in B_\delta(u_0)$ we have
$$
\|G(\la,u)-u_0\|_0\le \|G(\la,u)-G(\la,u_0)\|_0+ \|\d_uF(\la,u_0)^{-1}F(\la,u_0)\|_0
\le\frac{\delta}{2}+\frac{1}{c_0}\|F(\la,u_0)\|_0.
$$
But $F(\cdot,u_0)$ is continuous by 
  \reff{ksmooth}. Hence, there exists $\eps \in (0,\eps_0]$ such that for all 
$\la \in [-\eps,\eps]$ the map $G(\la,\cdot)$ is a strict contraction of $B_\delta(u_0)$. Then the Banach
 fixed point theorem yields that 
for all $\la \in [-\eps,\eps]$ there exists exactly one solution $u=\hat{u}(\la) \in  B_\delta(u_0)$ to \reff{abst}.
\end{proof}

For proving Theorem \ref{thm:IFT} it remains to prove the solution regularity property
\beq
\label{reg}
\hat{u}(\la) \in U_l \mbox{ for all } \la \in [-\eps,\eps]  \mbox{ and } l=1,2,\ldots
\ee
and the smooth dependence property
\beq
\label{smoothdep}
 \la \in [-\eps,\eps] \mapsto \hat{u}(\la) \in U_l  \mbox{ is $C^\infty$-smooth for all } l=0,1,\ldots.
\ee
In order to prove \reff{reg} and \reff{smoothdep}, we introduce maps $u_n: [-\eps,\eps] \to B_\delta(u_0)$ for $n=0,1,\ldots$ by
\beq
\label{undef}
u_0(\la):=u_0,\; u_{n+1}(\la):=G(\la,u_n(\la))=u_n(\la)
-\d_uF(\la,u_0)^{-1}F(\la,u_n(\la)).
\ee
Since the contraction constant of $G(\la,\cdot)$ is independent of $\la$ (cf. \reff{Gest}), we have
\beq
\label{uniform}
\|u_n(\la)-\hat{u}(\la)\|_0 \to 0 \mbox{ for } n \to \infty \mbox{ uniformly in } \la.
\ee
Hence,  because of a classical theorem from calculus (see, e.g. \cite[Theorem 8.6.3]{Dieudonne}), for proving \reff{reg} and  \reff{smoothdep}
it is sufficient to show that for all nonnegative integers $l$ and $n$
\beq
\label{reg1}
u_n(\la) \in U_l \mbox{ for all }  \la \in [-\eps,\eps] 
\ee
and
\beq
\label{smoothdep1}
\la \in  [-\eps,\eps] \mapsto u_n(\la) \in U_l \mbox{ is $C^\infty$-smooth } 
\ee
and that 
for all nonnegative integers $k$ and $l$
\beq
\label{smoothdep2}
u_n^{(k)}(\la) \mbox{ converges in } U_l \mbox{ for } n \to \infty \mbox{ uniformly  in } \la.
\ee

\begin{lem}
\label{step3}
For all  $\la \in [-\eps_0,\eps_0]$, $s \in \R$,   all positive integers $l$  and  all $u \in U_l$  we have $F(\la,u) \in U_l$ and 
\begin{eqnarray}
&&T(s)A^lF(\la,u)-\d_u\F(\la,s,T(s)u)T(s)A^lu\nonumber\\
&&=\sum_{k=2}^l\,\,\sum_{j_1+\ldots +j_k=l}c^0_{j_1 \ldots j_k}\d^k_u\F(\la,s,T(s)u)\left(T(s)A^{j_1}u,\ldots,T(s)A^{j_k}u\right)\nonumber\\
&&+\sum_{j=1}^{l-1}\,\sum_{k=1}^{l-j}\,\,\sum_{j_1+\ldots +j_k=l-j}c^j_{j_1 \ldots j_k}\d_s^j\d^{k}_u\F(\la,s,T(s)u)\left(T(s)A^{j_1}u,\ldots,T(s)A^{j_k}u\right)\nonumber\\
&&+\d_s^l\F(\la,s,T(s)u),
\label{mformula}
\end{eqnarray}
where $c^j_{j_1 \ldots j_k}$ are certain natural numbers not depending on $\la$, $s$ and $u$.
\end{lem}
\begin{proof}
Because of \reff{FFdef} we have
\beq
\label{invar}
T(s)F(\la,u)=\F(\la,s,T(s)u).
\ee
If $u \in U_l$, then $T(\cdot)u \in C^l(\R;U_0)$, and hence $\F(\la,s,T(\cdot)u) \in C^l(\R;U_0)$ due to \reff{infsmooth}. Then, \reff{invar}
yields $T(\cdot)F(\la,u)\in C^l(\R;U_0)$, i.e. $F(\la,u) \in U_l$. 

Given $u \in U_1$, we can differentiate \reff{invar} with respect to $s$ and get
$
T(s)AF(\la,u)-\d_u\F(\la,s,T(s)u)T(s)Au=\d_s\F(\la,s,T(s)u).
$
The formula \reff{mformula} with $l=1$ is therewith proved.

Now we do induction on $l$. Suppose  \reff{mformula} is true for a fixed $l$. 
Given $u \in U_{l+1}$, we can differentiate  \reff{mformula} with respect to $s$, and, using the chain rule,
 get \reff{mformula} with $l$ 
replaced by $l+1$.
\end{proof}

For brevity, for all positive integers $l$ we introduce  maps  $F_l:[-\eps_0,\eps_0]\times U_{l-1} \to U_0$ 
given by the right-hand side of  \reff{mformula} with $s=0$, i.e.
\begin{eqnarray*}
F_l(\la,u)&:=&
\sum_{k=2}^l\,\,\sum_{j_1+\ldots +j_k=l}c^0_{j_1 \ldots j_k}\d^k_uF(\la,u)\left(A^{j_1}u,\ldots,A^{j_k}u\right)\nonumber\\
&&+\sum_{j=1}^{l-1}\,\sum_{k=1}^{l-j}\,\,\sum_{j_1+\ldots +j_k=l-j}c^j_{j_1 \ldots j_k}\d_s^j\d^{k}_u\F(\la,0,u)\left(A^{j_1}u,\ldots,A^{j_k}u\right)\nonumber\\
&&+\d_s^l\F(\la,0,u).
\end{eqnarray*}
Then  \reff{mformula} with $s=0$ reads
\beq
\label{kEig1}
A^lF(\la,u)-\d_uF(\la,u)A^lu=F_l(\la,u) \mbox{ for all } u \in U_l.
\ee
Moreover, it follows directly from the assumptions \reff{infsmooth}--\reff{apriori} that
\beq
\label{minfsmooth}
F_l(\la,\cdot) \in C^\infty(U_{l-1};U_0),
\ee
\beq
\label{mksmooth}
\d_u^jF_l(\cdot,u)(u_1,\ldots,u_j) \in C^m([-\eps_0,\eps_0]; U_0) \mbox{ for all }
u,u_1,\ldots,u_j \in U_{l+m-1}
\ee
and  that there exists  $\tilde{d}_{jlm}>0$ such that for all $u, u_1,\ldots,u_l \in U_{l+m-1}$ with $\|u-u_0\|_{l+m-1} \le 1$
and for all  $\la \in [-\eps_0,\eps_0]$ it holds
\beq
\label{mapriori}
\|\d_\la^m\d_u^jF_l(\la,u)(u_1,\ldots,u_j)\|_0 \le  \tilde{d}_{jlm}
\|u_1\|_{l+m-1} \ldots \|u_j\|_{l+m-1}.
\ee

\begin{rem}\rm
\label{gain}
Due to \reff{kEig1}, the map $u\in U_l \mapsto A^lF(\la,u)-\d_uF(\la,u)A^lu$ can be extended to a smooth map from $U_{l-1}$ into $U_0$, 
although the maps  $u\in U_l \mapsto A^lF(\la,u)$ and  $u \in U_l \mapsto \d_uF(\la,u)A^lu$ cannot be extended  in such a way, in general.
This will be essentially used in the proofs of Lemmas \ref{lemreg2} and \ref{lemsmoothdep2} below.
Moreover, because of \reff{minfsmooth}, one can differentiate  the identity \reff{kEig1} in $u$ and get
\begin{eqnarray}
\label{kEig2}
&&A^l\d_uF(\la,u)v-\d_uF(\la,u)A^lv\nonumber\\
&&=\d_uF_l(\la,u)v+\d_u^2F(\la,u)(A^lu,v)=:K_l(\la,u)v \mbox{ for all } u,v \in U_l.
\end{eqnarray}
Here for all $\la \in [-\eps_0,\eps_0]$ and $u \in U_{l}$ the linear operator $K_l(\la,u)$, which is defined by \reff{kEig2},
is bounded from $U_{l-1}$ to $U_0$. In other words:  The linear bounded operators $v \in U_l \mapsto  A^l\d_uF(\la,u)v \in U_0$ and 
$v \in U_l \mapsto \d_uF(\la,u)A^lv \in U_0$ cannot be extended to 
linear bounded operators   from $U_{l-1}$ to $U_0$, in general, but their difference can  be extended in such a way.
Even more:  From \reff{apriori} and  \reff{mapriori} it follows that there exists  $\tilde{c}_{l}>0$ 
such that for all $u\in U_{l}$ with $\|u-u_0\|_{l} \le 1$
and for all  $\la \in [-\eps_0,\eps_0]$ it holds
\beq
\label{KAbsch}
\|K_{l}(\la,u)\|_{\LL(U_{l-1};U_0)} \le \tilde{c}_{l}.
\ee
\end{rem}

\begin{lem}
\label{apriori1}
(i) For all  all $\la \in [-\eps_0,\eps_0]$ and all nonnegative integers $l$  it holds $F(\la,\cdot) \in C^\infty(U_l;U_l)$.

(ii) For all nonnegative integers $j,l,m$ and all $u,u_1,\ldots,u_j \in U_{l+m}$ it holds 
$\d_u^jF(\cdot,u)(u_1,\allowbreak\ldots,u_j) \in C^m([-\eps_0,\eps_0];U_l)$.

(iii) For all   $\la \in [-\eps_0,\eps_0]$ and  all nonnegative integers $j,l,m$ there exists  $d_{jlm}>0$ such that for all $u, u_1,\ldots,u_j \in U_{l+m}$ with $ \|u-u_0\|_{l+m} \le 1$
 it holds 
\beq
\label{mkapriori}
\|\d_\la^m\d_u^jF(\la,u)(u_1,\ldots,u_j)\|_{l} \le d_{jlm}\|u_1\|_{l+m}\ldots\|u_j\|_{l+m}.
\ee
\end{lem}
\begin{proof}
Claim (i) follows from \reff{infsmooth}, \reff{kEig1} and \reff{minfsmooth}.

Differentiating the identity \reff{kEig1} $j$-times in $u$, we get
for all $\la \in  [-\eps_0,\eps_0]$ and $u,u_1,\ldots,u_j \in U_l$
\begin{eqnarray*}
\lefteqn{
A^l\d_u^jF(\la,u)(u_1,\ldots,u_j)-\sum_{k=1}^j\d_u^jF(\la,u)(u_1,\ldots,u_{k-1},A^lu_k,u_{k+1},\ldots,u_j)}\nonumber\\
&&=\d_u^jF_l(\la,u)(u_1,\ldots,u_j)+\d_u^{j+1}F(\la,u)(A^lu,u_1,\ldots,u_j). 
\end{eqnarray*}
Therefore Claim (ii) follows from \reff{ksmooth}.

Finally, for fixed  $u,u_1,\ldots,u_j \in U_{l+m}$ the last identity can be  differentiated $m$-times in $\la$, which gives
\begin{eqnarray*}
\lefteqn{
A^l\d_\la^m\d_u^jF(\la,u)(u_1,\ldots,u_j)-\sum_{k=1}^j\d_\la^m\d_u^jF(\la,u)(u_1,\ldots,u_{k-1},A^lu_k,u_{k+1},\ldots,v_j)}\nonumber\\
&&=\d_\la^m\d_u^jF_l(\la,u)(u_1,\ldots,u_j)+\d_\la^m\d_u^{j+1}F(\la,u)(A^lu,u_1,\ldots,u_j). 
\end{eqnarray*}
Therefore \reff{apriori} and \reff{mapriori} yield the Claim (iii).
\end{proof}

\begin{lem}
\label{lemln3}
For all nonnegative integers $l$ and all positive integers $m$ the map $(\la,u) \in  [-\eps_0,\eps_0] \times U_{l+m} \mapsto F(\la,u) \in U_l$ is  $C^m$-smooth.
\end{lem}
\begin{proof}
We use the so-called converse to Taylor's theorem (see, e.g.,  \cite[Supplement 2.4B]{Abraham}).
This theorem claims that the lemma is true if there exist continuous maps 
$$
G_{jk}: [-\eps_0,\eps_0] \times U_{l+m} \to {\cal S}_j(U_{l+m};U_l) \mbox{ and } R_j: \left([-\eps_0,\eps_0] \times U_{l+m}\right)^2 \to {\cal S}_j(U_{l+m};U_l)
$$
with 
\beq
\label{Rnull}
R_j(\la,u,0,0)=0
\ee
and
\beq
\label{covT}
F(\la+\mu,u+v)-F(\la,u)=\sum_{j=1}^m\sum_{k=0}^{m-j}\mu^{k}G_{jk}(\la,u)v^j+\sum_{j=0}^m\mu^{m-j}R_j(\la,u,\mu,v)v^j.
\ee
As usual, here we denote by $G_{jk}(\la,u)v^j \in  U_{l}$  
the element which one gets by applying the multilinear operator $G_{jk}(\la,u)$  
to the tuple $(v,\ldots,v) \in U_{l+m}^j$.
Similar notation rule is also used for $R_j(\la,u,\mu,v)v^j$ and will be used below.

Because of the assumption \reff{infsmooth} we can apply the classical (direct) Taylor theorem  (cf., e.g., \cite[Theorem 2.4.15]{Abraham})
to the map $F(\la+\mu,\cdot):U_0 \to U_0$, and we get
\begin{eqnarray*}
\lefteqn{
F(\la+\mu,u+v)-\sum_{j=0}^m\frac{1}{j!}\d_u^jF(\la+\mu,u)v^j}\\
&&=\frac{1}{(m-1)!}\int_0^1(1-t)^{m-1}\left(\d_u^{m}F(\la+\mu,u+tv)-\d_u^{m}F(\la+\mu,u)\right)v^{m} dt.
\end{eqnarray*}
Similarly, on the account of the assumption \reff{ksmooth}, we can apply the Taylor theorem  
to the map $\d_u^jF(\cdot,u)v^j: [-\eps_0,\eps_0] \to U_0$ for $v \in U_m$ and $j=0,1,\ldots,m-1$, and we get
\begin{eqnarray*}
\lefteqn{
\d_u^jF(\la+\mu,u)v^j-\sum_{k=0}^{m-j}\frac{\mu^k}{k!}\d_\la^k\d_u^jF(\la,u)v^j}\\
&&=\frac{\mu^{m-j}}{(m-j-1)!}
\int_0^1(1-t)^{m-j-1}\left(\d_\la^{m-j}\d_u^{j}F(\la+t\mu,u)-\d_\la^{m-j}\d_u^{j}F(\la,u)\right)v^{j} dt.
\end{eqnarray*}
This yields \reff{covT} with
\beq
\label{Gjdef}
G_{jk}(\la,u)(v_1,\ldots,v_j):=\frac{1}{j!k!}\d_\la^{k}\d_u^{j}F(\la,u)(v_1,\ldots,v_j),
\ee
\begin{eqnarray}
&&R_j(\la,u,\mu,v)(v_1,\ldots,v_j)\nonumber\\
&&:=\frac{1}{(m-j-1)!j!}\int_0^1(1-t)^{m-j-1}\left(\d_\la^{m-j}\d_u^{j}F(\la+t\mu,u)-\d_\la^{m-j}\d_u^{j}F(\la,u)\right)(v_1,\ldots,v_j) dt\nonumber\\
&&\mbox{ for } j=0,1,2,\ldots,m-1
\label{Rmdef}
\end{eqnarray}
and
\begin{eqnarray}
&&R_m(\la,u,\mu,v)(v_1,\ldots,v_m)\nonumber\\
&&:=\frac{1}{(m-1)!}\int_0^1(1-t)^{m-1}\left(\d_u^{m}F(\la+\mu,u+tv)-\d_u^{m}F(\la,u)\right)(v_1,\ldots,v_m) dt.
\label{Rdef}
\end{eqnarray}
The property  \reff{Rnull} follows directly from the definition \reff{Rmdef}. On the account of 
 \reff{mkapriori}, the multilinear operators defined by \reff{Gjdef}--\reff{Rdef} are bounded from
$U_{l+m}^j$ to $U_l$ and from $U_{l+m}^m$ to $U_l$, respectively.
Moreover,  the map $(\la,u) \in [-\eps_0,\eps_0] \times U_{l+m} \mapsto G_j(\la,u) \in  {\cal S}_j(U_{l+m};U_l)$  is continuous.
Indeed, 
\begin{eqnarray*}
\lefteqn{
\left(G_{jk}(\la+\mu,u+v)-G_{jk}(\la,u)\right)
\left(v_1,\ldots,v_j\right)}\\
&&=\frac{1}{j!k!}\int_0^1\left(\d_\la^{k}\d_u^{j+1}F(\la+\mu,u+tv)(v,v_1,\ldots,v_j)+\mu\d_\la^{k+1}\d_u^{j}F(\la+t\mu,u)(v_1,\ldots,v_j)\right)dt.
\end{eqnarray*}
Since $j+k \le m$ and $j \ge 1$ (cf. \reff{covT}),  we have $k \le m-1$. Hence, we get from \reff{mkapriori} 
the convergency
$$
\|G_{jk}(\la+\mu,u+v)-G_{jk}(\la,u)\|_{\LL(U_{l+m};U_l)} \to 0 \mbox{ for } |\mu|+\|v\|_{l+m} \to 0.
$$
Here we used that $A$ is a closed linear operator in $U_0$, and hence $A\int_0^1f(t)dt=\int_0^1Af(t)dt$ for any continuous function $f:[0,1] \to U_1$.

Similarly one shows that for $j=1,2,\ldots,m$  the maps $(\la,u,\mu,v) \in 
\left([-\eps_0,\eps_0]\times U_{l+m}\right)^2 \mapsto R_j(\la,u,\mu,v) \in  {\cal S}_j(U_{l+m};U_l)$  are continuous.
In the remaining case $j=0$ we have to show (cf. \reff{Rmdef}) that the map
$$
(\la,\mu,u) \in [-\eps_0,\eps_0]^2 \times U_{l+m} \mapsto
\int_0^1(1-t)^{m-1}\left(\d_\la^{m}F(\la+t\mu,u)-\d_\la^{m}F(\la,u)\right) dt \in U_l
$$
is continuous. But this follows from the following two facts: First, because of  \reff{mkapriori} we have
$$
\|\d_\la^{m}F(\la+t\mu,u+v)-\d_\la^{m}F(\la+t\mu,u)\|_l=\left\|\int_0^1\d_\la^{m}\d_uF(\la+t\mu,u+sv)vds\right\|_l
\le \mbox{const}\|v\|_{l+m},
$$
where the constant does not depend on $\la, t, \mu$ and $u$.
And, second, due to Lemma \ref{apriori1} (ii),
the map $\d_\la^{m}F(\cdot,u):[-\eps_0,\eps_0] \to U_l$
is continuous for all $u \in U_{l+m}$.
\end{proof}

\begin{lem}
\label{lemu0}
For all nonnegative integers $l$ we have $u_0 \in U_l$.
\end{lem}
\begin{proof}
Thanks to the  assumption \reff{zero}, for all $s \in \R$ the element $u=T(s)u_0$ is a solution to the equation
$
\F(0,s,u)=T(s)F(0,T(-s)u)=0.
$
Further, because of the assumption \reff{infsmooth} we have  $\F(0,\cdot,\cdot) \in C^\infty(\R \times U_0;U_0)$. 
Finally,  \reff{FH} and \reff{coerz} give that 
the operator $\d_u\F(0,0,u_0)=\d_uF(0,u_0)$ is bijective from $U_0$ to $U_0$.
Hence, the classical implicit function theorem implies that the map $s \approx 0 \mapsto T(s)u_0 \in U_0$ is $C^\infty$-smooth.
This yields the claim.
\end{proof}

\begin{lem}
\label{lem-1}
For all nonnegative integers $l$ there exist  $c_{l}>0$ such that for  all  $\la \in [-\eps_0,\eps_0]$ 
and all $f\in U_{l}$ 
 it holds $\d_uF(\la,u_0)^{-1}f \in U_l$ and
\beq
\label{lapriori}
\|\d_uF(\la,u_0)^{-1}f\|_{l} \le c_{l}\|f\|_{l}.
\ee
\end{lem}
\begin{proof}
The operator $\d_u\F(\la,s,T(s)u_0)=T(s)\d_uF(\la,u_0)T(-s)$ is an isomorphism from $U_0$ to $U_0$, and the map
$s \in \R \mapsto \d_u\F(\la,s,T(s)u_0) \in {\cal L}(U_0)$ is $C^\infty$-smooth because of \reff{infsmooth}
and  Lemma \ref{lemu0}. Hence, for any $f \in U_l$ the map
$s \in \R \mapsto T(s)\d_uF(\la,u_0)^{-1}f=\d_u\F(\la,s,T(s)u_0)^{-1}T(s)f \in U_0$
is $C^l$-smooth, i.e. $\d_uF(\la,u_0)^{-1}f\in U_l$.

Now, let us prove \reff{lapriori} by
induction on $l$. For $l=0$ this follows from the assumption \reff{coerz}.

Let us do the induction step. Suppose that \reff{lapriori} is true for an arbitrary fixed $l$. Take $f \in U_{l+1}$.
From \reff{kEig2} it follows 
\beq
\label{kEig3}
A^{l+1}\d_uF(\la,u_0)^{-1}f=\d_uF(\la,u_0)^{-1}\left(A^{l+1}f-K_{l+1}(\la,u_0)\d_uF(\la,u_0)^{-1}f\right).
\ee
Therefore  \reff{KAbsch} and the induction assumption imply
\begin{eqnarray*}
&&\|A^{l+1}\d_uF(\la,u_0)^{-1}f\|_0\le c_0\left(\|A^{l+1}f\|_0+\|K_{l+1}(\la,u_0)\d_uF(\la,u_0)^{-1}f\|_0\right)\\
&&\le  c_0\left(\|f\|_{l+1}+\tilde{c}_{l+1}\|F(\la,u_0)^{-1}f)\|_l\right)\le  c_0\left(\|f\|_{l+1}+\tilde{c}_{l+1}c_l\|f\|_{l}\right).
\end{eqnarray*}
\end{proof}

\begin{lem}
\label{lemreg1}
For all nonnegative integers $l$ and $n$ it holds \reff{reg1}.
\end{lem}
\begin{proof}
We do induction on $n$.  
For $n=0$ the claim is true because of  Lemma \ref{lemu0}.
In order to do the induction step, let us suppose that $u_n(\la)$  belongs to all spaces $U_l$. Then Lemma \ref{step3} yields
$F(\la,u_n(\la)) \in U_l$, and Lemma \ref{lem-1} implies $u_{n+1}(\la)=u_{n}(\la)-\d_uF(\la,u_0)^{-1}F(\la,u_n(\la)) \in U_l$.
\end{proof}

\begin{lem}
\label{lemln2}
For all nonnegative integers $l$ and $m$ 
\beq
\label{inver}
\d_uF(\cdot,u_0)^{-1}f\in C^m([-\eps_0,\eps_0];U_l) \mbox{ for all } f \in U_{l+m}
\ee
and there exist $c_{lm}>0$ such that for all $\la \in [-\eps_0,\eps_0]$ and $f \in U_{l+m}$ it holds
\beq
\label{mapri}
\left\|\frac{d^m}{d\la^m}\left[\d_uF(\la,u_0)^{-1}f\right]\right\|_l \le c_{lm}\|f\|_{l+m}.
\ee
\end{lem}
\begin{proof}
We do induction on $m$. For brevity set $A(\la):=\d_uF(\la,u_0)$.

First consider the case $m=0$. For $f \in U_{l}$ we have
\begin{eqnarray*}
&&\left\|\left(A(\la+\mu)^{-1}-A(\la)^{-1}\right)f\right\|_l=\left\|A(\la+\mu)^{-1}\left(A(\la)-A(\la+\mu)\right)A(\la)^{-1}f\right\|_l\nonumber\\
&&\le c_l\left\|\left(A(\la)-A(\la+\mu)\right)A(\la)^{-1}f\right\|_l \to 0 \mbox{ for } \mu \to 0
\end{eqnarray*}
because of Lemma \ref{apriori1} (ii) with $j=1, m=0$. Moreover, \reff{mapri} with $m=0$ is proved in Lemma \ref{lem-1}.

Now we  consider the case $m=1$. Take $f \in U_{l+1}$. Then
\begin{eqnarray*}
&&\left(A(\la+\mu)^{-1}-A(\la)^{-1}+\mu A(\la)^{-1}A'(\la)A(\la)^{-1}\right)f\nonumber\\
&&=A(\la+\mu)^{-1}\left(A(\la)-A(\la+\mu)+\mu A'(\la)\right)A(\la)^{-1}f\nonumber\\
&&+\mu\left(A(\la+\mu)^{-1}-A(\la)^{-1}\right)A'(\la)A^{-1}(\la)f.
\end{eqnarray*}
Here $A'(\la):=\d_\la\d_uF(\la,u_0) \in \LL(U_{l+1};U_l)$ is the derivative of $A(\la)$ with respect to $\la$.
The $U_l$-norm of the first term in the right-hand side is $o(\mu)$ for $\mu \to 0$ because of  Lemma~\ref{apriori1} (ii) with $j=m=1$. And the $U_l$-norm of the 
second term in the right-hand side is $o(\mu)$ for $\mu \to 0$ 
because of the first step ($m=0$) of this proof. Hence, $A(\cdot)^{-1}f$ is differentiable as a map from $[-\eps_0,\eps_0]$ to $U_l$, and the derivative in $\la$ is $-A(\la)^{-1}A'(\la)A(\la)^{-1}f$.
Moreover, \reff{mapri} with $m=1$ 
follows from \reff{mkapriori} and \reff{lapriori}. Similarly one shows the continuity of the derivative:
\begin{eqnarray}
&&\left\|\left(A(\la+\mu)^{-1}A'(\la+\mu)A(\la+\mu)^{-1}-A(\la)^{-1}A'(\la)A(\la)^{-1}\right)f\right\|_l\nonumber\\
&&\le c_ld_{1l1}\left\|\left(A(\la+\mu)^{-1}-A(\la)^{-1}\right)f\right\|_l\nonumber\\
&&+c_l\left\|\left(A'(\la+\mu)-A'(\la)\right)A(\la)^{-1}f\right\|_l\nonumber\\
&&+\left\|\left(A(\la+\mu)^{-1}-A(\la)^{-1}\right)A'(\la)A(\la)^{-1}f\right\|_l.
\label{conti}
\end{eqnarray}
This tends to zero  because of the case $m=0$ and because of Lemma \ref{apriori1} (ii) with $j=m=1$.

Now we do the induction step from $m$ to $m+1$. Suppose that for a fixed $m \ge 1$ and all $l$ the assertions of the lemma are true.
For $k=0,1,\ldots,m$ and $l=0,1,2,\ldots$ we consider the map
\beq
\label{Bkdef}
B_k^l:([-\eps_0,\eps_0]\times U_{k+l} \to U_l: \; B_k^l(\la,g):=\frac{d^k}{d\la^k}A(\la)^{-1}g.
\ee
Due to the induction assumption we have
\beq
\label{indas}
\|B_k^l(\la,g)\|_l \le c_{lk}\|g\|_{k+l}.
\ee

Take $f \in U_{l+m+1}$. We have to show that $A(\cdot)^{-1}f \in C^{m+1}([-\eps_0,\eps_0];U_l)$ or, what is the same, that
$\la \in [-\eps_0,\eps_0] \mapsto A(\la)^{-1}A'(\la)A(\la)^{-1}f \in U_l$ is $C^m$-smooth.
For that it is sufficient to show that
\beq
\label{indu}
(\la,\mu,\nu) \in [-\eps_0,\eps_0]^3 \mapsto A(\la)^{-1}A'(\mu)A(\nu)^{-1}f \in U_l \mbox{ is $C^m$-smooth.}
\ee
Because of Lemma \ref{apriori1} (ii) and  the induction assumption, the
 partial derivative $\partial_\la^i\partial_\mu^j\partial_\nu^k$ 
of the map \reff{indu} with $i+j+k \le m$ exists and is equal to 
\beq
\label{indu1}
(\la,\mu,\nu) \in [-\eps_0,\eps_0]^3 \mapsto B_i^l(\la,A^{(j+1)}(\mu)B_k^{l+m+1-k}(\nu,f))\in U_l.
\ee
But \reff{indu1} is a superposition of the following three continuous maps: First take
$$
(\la,\mu,\nu) \in [-\eps,\eps]^3 \mapsto (\la,\mu,\nu,B_k^{l+m+1-k}(\nu,f)) \in  [-\eps,\eps]^3 \times U_{l+m+1-k}, 
$$
then take
$$
(\la,\mu,\nu,g) \in [-\eps,\eps]^3 \times  U_{l+m+1-k}
\mapsto   (\la,\mu,\nu,A^{(j+1)}(\mu)g) \in  [-\eps,\eps]^3 \times U_{l+m-j-k}, 
$$
and finally take
$$
(\la,\mu,\nu,h)  \in  [-\eps,\eps]^3 \times U_{l+m-j-k}\hookrightarrow U_{i+l}  \mapsto   B_i^l(\la,h) \in U_{l}. 
$$
This means that  the map \reff{indu1} is continuous, hence,  \reff{indu} is true.

It remains to show \reff{mapri} with $m$ replaced by $m+1$. In order to show this we calculate (using the product rule)
\begin{eqnarray*}
\frac{d^{m+1}}{d\la^{m+1}}\left[A(\la)^{-1}f\right]&=&-\frac{d^{m}}{d\la^{m}}\left[A(\la)^{-1}A'(\la)A(\la)^{-1}f\right]\\
&=&-\sum_{i+j+k=m}B_i^l(\la,A^{(j+1)}(\la)B_k^{l+m+1-k}(\la,f)).
\end{eqnarray*}
But  \reff{mkapriori} and \reff{indas} imply (for $i+j+k = m$ and $f \in U_{l+m+1}$)
$$
\|B_i^l(\la,A^{(j+1)}(\la)B_k^{l+m+1-k}(\la,f))\|_l \le \mbox{const}\|f\|_{l+m+1},
$$
where the constant does not depend on $\la$ and $f$.
\end{proof}

\begin{lem}
\label{lemsmoothdep1}
For all nonnegative integers $l$ and $n$ it holds \reff{smoothdep1}.
\end{lem}
\begin{proof}
We do induction on $n$. 
Obviously, for $n=0$  assertion \reff{smoothdep1} is true.

Let us do the induction step. Suppose that  \reff{smoothdep1} is true 
for a fixed $n$. On the account of \reff{undef},
the induction step will be done if we show that for all nonnegative integers $l$ the map
$\la \in  [-\eps,\eps] \mapsto \d_uF(\la,u_0)^{-1}F(\la,u_n(\la)) \in U_l$ is $C^\infty$-smooth. 
For that it is sufficient to show that all partial derivatives of the map
\beq
\label{map1}
(\la,\mu) \in [-\eps_0,\eps_0]^2 \mapsto  \d_uF(\la,u_0)^{-1}F(\mu,u_n(\mu)) \in U_l
\ee
exist and are continuous. 

Because of Lemma \ref{lemln3} and of the induction assumption, for all $l$ the map $\mu \in [-\eps_0,\eps_0] \mapsto  F(\mu,u_n(\mu)) \in U_l$
is $C^\infty$-smooth.  Denote by $C^l_m(\mu)$ the $m$-th derivative of this map.
Then Lemma \ref{lemln2} yields that all partial derivatives $\partial_\la^j\partial_\mu^m$ of the map \reff{map1} exist and equal to $(\la,\mu) \in [-\eps_0,\eps_0]^2 \mapsto B^l_j(\la,C^{j+l}_m(\mu))
\in U_l$. Here we used the notation \reff{Bkdef}. But this map is continuous because it is the superposition of the following two continuous maps: 
\begin{eqnarray*}
(\la,\mu) \in [-\eps,\eps]^2 &\mapsto & (\la,\mu,C^{j+l}_m(\mu)) \in  [-\eps,\eps]^2 \times U_{j+l}, \\
(\la,\mu,u)  \in  [-\eps,\eps]^2 \times U_{j+l} &\mapsto &  B^l_j(\la,u) \in U_l.
\end{eqnarray*}
\end{proof}

\begin{lem}
\label{lemreg2}
For all  nonnegative integers $l$ it holds
\beq
\label{reg2}
u_n(\la) \mbox{ converges in } U_l  \mbox{ for } n \to \infty \mbox{ uniformly  in } \la \in [-\eps,\eps].
\ee
\end{lem}
\begin{proof}
We do induction on $l$. 
Because of \reff{uniform} the assertion \reff{reg2} is true for $l=0$.

Now let us do the induction step. Suppose that  \reff{reg2} is true for a certain  $l$.
We have to show that 
\beq
\label{reg2a}
A^{l+1}u_n(\la)  \mbox{ converges in } U_{0}  \mbox{ for } n \to \infty \mbox{ uniformly in } \la \in [-\eps,\eps].
\ee
Let us introduce operators $L_{l+1}(\la) \in \LL(U_l;U_0)$ by 
$$
L_{l+1}(\la)f:=\d_uF(\la,u_0)^{-1}K_{l+1}(\la,u_0)\d_uF(\la,u_0)^{-1}f.
$$
Then \reff{undef}, \reff{kEig1} and \reff{kEig3} imply
\begin{eqnarray*}
&&A^{l+1}\left((u_{n+1}(\la)-u_n(\la)\right)\\
&&=-\d_uF(\la,u_0)^{-1}\left(\d_uF(\la,u_n(\la))A^{l+1}u_n(\la)+F_{l+1}(\la,u_n(\la))\right)+L_{l+1}(\la)F(\la,u_n(\la)).
\end{eqnarray*}
This yields
\begin{eqnarray}
\label{rlhs}
&&A^{l+1}u_{n+1}(\la)-\left(I-\d_uF(\la,u_0)^{-1}\d_uF(\la,u_n(\la))\right)A^{l+1}u_n(\la)\nonumber\\
&&=-\d_uF(\la,u_0)^{-1}F_{l+1}(\la,u_n(\la))+L_{l+1}(\la)F(\la,u_n(\la)).
\end{eqnarray}
Moreover,  \reff{est} we yields
\beq
\label{Ban}
\|I-\d_uF(\la,u_0)^{-1}\d_uF(\la,u_n(\la))\|_{\LL(U_0)} \le \frac{1}{2} \mbox{ for all } \la \in [-\eps,\eps].
\ee
Hence, Lemma \ref{A2} implies that \reff{reg2a} is true if  the right-hand side of \reff{rlhs} converges in $U_0$ for $n \to \infty$
uniformly in $\la$.

Let us show that the right-hand side of \reff{rlhs} indeed converges in $U_0$ for $n \to \infty$
uniformly in $\la$. 
From \reff{mkapriori}, the induction assumption and
$$
F(\la,u_n(\la))-F(\la,u_m(\la))=
\int_0^1\d_uF(\la,su_n(\la)+(1-s)u_m(\la))\left(u_n(\la)-u_m(\la)\right)ds,
$$
it follows that $F(\la,u_n(\la))$ converges in $U_l$ for $n \to \infty$
uniformly in $\la$. 
On the other hand,  \reff{KAbsch} and \reff{lapriori} yield 
$$
\|L_{l+1}(\la)\|_{\LL(U_l:U_0)} \le\mbox{const} \mbox{ for all } \la \in [-\eps_0,\eps_0].
$$
Hence,  $L_{l+1}(\la)F(\la,u_n(\la))$ converges in $U_l$ for $n \to \infty$
uniformly in $\la$.

We proced similarly with the remaining term in the right-hand side of  \reff{rlhs}. Because of \reff{mapriori} and \reff{lapriori} we have
\begin{eqnarray*}
&&\|\d_uF(\la,u_0)^{-1}\left(F_{l+1}(\la,u_n(\la))-F_{l+1}(\la,u_m(\la))\right)\|_0\\
&&=\left\|\d_uF(\la,u_0)^{-1}\int_0^1\d_uF_{l+1}(\la,su_n(\la)+(1-s)u_m(\la))\left(u_n(\la)-u_m(\la)\right)ds\right\|_0\\
&&\le \mbox{const}\|u_n(\la)-u_m(\la)\|_l,
\end{eqnarray*}
which tends to zero for $m,n \to \infty$ uniformly in $\la$ by the induction assumption.
\end{proof}

\begin{lem}
\label{lemsmoothdep2}
For all nonnegative integers $k$ and $l$ it holds \reff{smoothdep2}.
\end{lem}
\begin{proof}
We have to show that for all nonnegative integers $l$ and $m$ the sequence $A^lu_n^{(m)}(\la)$ converges in $U_0$ for $n \to \infty$ uniformly in $\la$, and we do this by induction on $m$.
For $m=0$ the claim is true  because of Lemma \ref{lemreg2}.

Let us do the induction step $m \to m+1$. The induction assumption is that for a fixed $m \in \N$ we have that
\beq
\label{smoothdep2a}
\begin{array}{l}
\mbox{for all } k=1,\ldots,m \mbox{ and  nonnegative integers } l \mbox{ the sequence }
A^lu_n^{(k)}(\la)\\
\mbox{converges in } U_0  \mbox{ for } n \to \infty \mbox{ uniformly  with respect to } \la.
\end{array}
\ee
We have to show that 
\beq
\label{smoothdep2b}
\begin{array}{l}
\mbox{for all nonnegative integers  } l \mbox{ the sequence }
A^lu_n^{(m+1)}(\la)\\
\mbox{converges in } U_0  \mbox{ for } n \to \infty \mbox{ uniformly  in } \la.
\end{array}
\ee
In order to prove \reff{smoothdep2b},
we differentiate   the identity $\d_uF(\la,u_0)(u_{n}(\la)-u_{n+1}(\la))=F(\la,u_n(\la))$ $m+1$ times
in $\la$. We get
\begin{eqnarray*}
&&\sum_{j=0}^{m+1}{m+1 \choose j}\d_\la^j\d_uF(\la,u_0)\left(u_{n}^{(m+1-j)}(\la)-u_{n+1}^{(m+1-j)}(\la)\right)\nonumber\\
&&=\frac{d^{m+1}}{d\la^{m+1}}F(\la,u_n(\la)))
=\frac{d^{m}}{d\la^{m}}\left(\d_\la F(\la,u_n(\la))+\d_u F(\la,u_n(\la))u'_n(\la)\right).
\end{eqnarray*}
Writing the terms with derivatives of order $m+1$ in the left-hand side and all other terms in the right-hand
 side, we get 
\begin{eqnarray}
&&\d_uF(\la,u_0)u_{n+1}^{(m+1)}(\la)-\left(\d_uF(\la,u_0)-\d_uF(\la,u_n(\la))\right)u_{n}^{(m+1)}(\la)\nonumber\\
&&=G(\la,u_n(\la),u_{n+1}(\la),u_n'(\la),u_{n+1}'(\la),\ldots,u_n^{(m)}(\la),u_{n+1}^{(m)}(\la))
\label{Gm}
\end{eqnarray}
with a certain map $G : [-\eps,\eps] \times U_{l+m+1}^{2(m+1)} \to U_l$. For example, for $m=1$ we have
\begin{eqnarray*}
&&G(\la,u,v,u_1,v_1) = \d_\la^2 F(\la,u)+2 \d_\la \d_uF(\la,u)u_1+\d_u^2 F(\la,u)(u_1,u_1),\\
&&+2\d_\la \d_uF(\la,u_0)(u_1-v_1)+\d_\la^2 \d_uF(\la,u_0)(u-v)
\end{eqnarray*}
and for $m=2$ we have
\begin{eqnarray*}
&&G(\la,u,v,u_1,v_1,u_2,v_2) =  \d_\la^3 F(\la,u)+3 \d_\la^2\d_u F(\la,u)u_1+3 \d_\la\d_u^2 F(\la,u)(u_1,u_1)\\
&&+\d_u^3 F(\la,u)(u_1,u_1,u_1)+3\d_u^2 F(\la,u)(u_1,u_2)\\
&&+3\d_\la \d_uF(\la,u_0)(u_2-v_2)+3\d_\la^2 \d_uF(\la,u_0)(u_1-v_1)+\d_\la^3 \d_uF(\la,u_0)(u-v).
\end{eqnarray*}
Because of \reff{mkapriori} the maps $G(\la,\cdot): U_{l+m+1}^{2(m+1)} \to U_l$ are Lipschitz continuous on bounded sets, with  Lipschitz constants not depending on $\la$. Hence \reff{smoothdep2a}
yields
\beq
\label{smoothdep2c}
\begin{array}{l}
\mbox{for all } l=0,1,2,\ldots \mbox{ the sequence } G(\la,u_n(\la),u_{n+1}(\la),\ldots,u_n^{(m)}(\la),u_{n+1}^{(m)}(\la))\\
\mbox{converges in } U_l  
\mbox{ for } n \to \infty \mbox{ uniformly in } \la.
\end{array}
\ee

Let us prove \reff{smoothdep2b} by induction on $l$. In order to consider the case $l=0$, we rewrite
\reff{Gm} as
\begin{eqnarray}
&&u_{n+1}^{(m+1)}(\la)-\left(I-\d_uF(\la,u_0)^{-1}\d_uF(\la,u_n(\la))\right)u_{n}^{(m+1)}(\la)\nonumber\\
&&=\d_uF(\la,u_0)^{-1}G(\la,u_n(\la),u_{n+1}(\la),\ldots,u_n^{(m)}(\la),u_{n+1}^{(m)}(\la)).
\label{Rn}
\end{eqnarray}
Because of  \reff{lapriori} and  \reff{smoothdep2c} with $l=0$ the right-hand side of \reff{Rn} converges in $U_0$ for $n \to \infty$ uniformly in $\la$.
Hence, \reff{Ban} and Lemma \ref{A2} imply that \reff{smoothdep2b} with $l=0$ is true.

For the induction step from $l$ to $l+1$ in  \reff{smoothdep2b} the induction assumption is that
for a  fixed $l$ we have 
\beq
\label{indvor1}
\begin{array}{l}
\mbox{for } j=0,1,\ldots,l \mbox{ the sequence }
A^ju_n^{(m+1)}(\la) \mbox{ converges }\\
\mbox{in $U_0$  for $n \to \infty$ uniformly  in  $\la$},
\end{array}
\ee
and the induction claim reads
\beq
\label{indclaim}
A^{l+1}u_n^{(m+1)}(\la) \mbox{ converges in $U_0$  for $n \to \infty$ uniformly in  $\la$}.
\ee
The  two induction assumptions \reff{smoothdep2a} and \reff{indvor1} together yield
\beq
\label{indvor}
\begin{array}{l}
A^ju_n^{(k)}(\la)  \mbox{ converges in $U_0$  for $n \to \infty$ uniformly  in $\la$ \mbox{ for }}\\
\mbox{$j=0,1,\ldots$, $k=0,1,\ldots,m$ and for $j=0,1,\ldots,l$, $k=m+1$.}
\end{array}
\ee
In order  to show that \reff{indvor} implies  \reff{indclaim},
we apply $A^{l+1}$ to 
\reff{Gm}, use \reff{kEig2}  
and get
\begin{eqnarray}
&&\d_uF(\la,u_0)A^{l+1}u_{n+1}^{(m+1)}(\la)+\left(\d_uF(\la,u_n(\la))-\d_uF(\la,u_0)\right)A^{l+1}u_{n}^{(m+1)}(\la)
\nonumber\\
&&=K_{l+1}(\la,u_n(\la))u_{n}^{(m+1)}(\la)+K_{l+1}(\la,u_0)\left(u_{n+1}^{(m+1)}(\la)-u_{n}^{(m+1)}(\la)\right)\nonumber\\
&&+A^{l+1}G(\la,u_n(\la),u_{n+1}(\la),\ldots,u_n^{(m)}(\la),u_{n+1}^{(m)}(\la)).
\label{Sn}
\end{eqnarray}
Because of \reff{mapriori} and \reff{mkapriori} the map 
$$
(u,v) \in U_{l+1} \times U_l \mapsto K_{l+1}(\la,u)v=\d_uF_l(\la,u)+\d_u^2F(\la,u)(A^lu,v) \in U_l
$$
(cf. \reff{kEig2}) is Lipschitz continuous on bounded sets, with a Lipschitz constnat not depending on $\la$.
Therefore $ K_{l+1}(\la,u_n(\la))u_n^{(m+1)}(\la)$ converges in $U_0$ for $n \to \infty$ uniformly in $\la$.
Hence, \reff{Gm} and \reff{indvor}   imply that the right-hand side of \reff{Sn} converges in $U_0$ for $n \to \infty$ uniformly in $\la$.
Therefore, as above, \reff{Ban} and Lemma \ref{A2} imply \reff{indclaim}.
\end{proof}

\section{Time-periodic solutions to first-order hyperbolic systems}
\label{Proof}
\renewcommand{\theequation}{{\thesection}.\arabic{equation}}
\setcounter{equation}{0}

\subsection{Setting and main result}
\label{setting2}

In this section we consider the boundary value problem (\ref{eq:1.1})--(\ref{eq:1.2}) where
 $m$ and $n$ are integers with $1 \le m \le n-1$,  $\eps_0>0$ is fixed, and the coefficient
functions $a_j : [0,1]\times [-\eps_0,\eps_0]\to\R$, $b_j : \R\times [0,1]\times [-\eps_0,\eps_0]\times \R\to\R$
and $r_{jk}:\R\times [-\eps_0,\eps_0] \to \R$ are supposed to be $C^\infty$-smooth. 
Moreover, 
$b_j$ and $r_{jk}$ are $2\pi$-periodic in $t$, and we suppose that
\beq \label{hyp0}
a_j(x,\la)\ne 0  \mbox{ for all }  x\in[0,1], \la \in  [-\eps_0,\eps_0] \mbox{ and } j=1,\dots,n,
\ee 
\beq \label{hyp}
a_j(x,\la)\ne a_k(x,\la) \mbox{ for all } x\in[0,1], \la \in  [-\eps_0,\eps_0]  \mbox{ and } 1\le j\ne k\le n.
\ee
Speaking about solutions to  (\ref{eq:1.1})--(\ref{eq:1.2}),  
we will mean classical solutions, i.e., $C^1$-smooth maps $u: \R \times [0,1] \to \R^n$ which satisfy  (\ref{eq:1.1})--(\ref{eq:1.2})
pointwise.

Let $u_0=(u_{01},\ldots,u_{0n})$ be a 
solution to  (\ref{eq:1.1})--(\ref{eq:1.2}) with $\la=0$. Set
\beq 
\label{b0def}
b^0_{j}(t,x,\la):=\partial_{u_j}b_{j}(t,x,\la,u_0(t,x)),
\ee
$$
R:=\max_{1 \le j \le m} \sum_{k=m+1}^n \sum_{l=1}^m \max_{t,s}|r_{jk}(t,0)r_{kl}(s,0)| \exp\int_0^1  \max_{t,s}
\left(\frac{b_k^0(t,x,0)}{a_k(x,0)}- \frac{b_j^0(s,x,0)}{a_j(x,0)}\right)dx,
$$
$$
S:=\max_{m+1 \le j \le n} \sum_{k=1}^m \sum_{l=m+1}^n \max_{t,s}|r_{jk}(t,0)r_{kl}(s,0)| \exp\int_0^1  \max_{t,s}
\left(\frac{b_j^0(t,x,0)}{a_j(x,0)}- \frac{b_k^0(s,x,0)}{a_k(x,0)}\right)dx.
$$
Suppose
\beq
\label{Fred}
R<1 \mbox{ or } S<1.
\ee 
In the particular case $m=1,n=2$ we have $R=S$, and in this case we set
\begin{eqnarray*}
&&\displaystyle R_0(t):=\left|r_{12}(t,0)r_{21}\left(t+\int_0^1\frac{d \xi}{a_2(\xi,0)},0\right)\right|\\
&&\displaystyle \times \exp \int_0^1 \left(\frac{1}{a_2(x,0)}
b_2^0\left(t+\int_0^x\frac{d \xi}{a_2(\xi,0)},x,0\right)-\frac{1}{a_1(x,0)}
b_1^0\left(t+\int_0^x\frac{d \xi}{a_1(\xi,0)},x,0\right)\right)dx,\nonumber\\
\end{eqnarray*}
\begin{eqnarray*}
&&\displaystyle S_0(t):=\left|r_{21}(t,0)r_{12}\left(t-\int_0^1\frac{d \xi}{a_1(\xi,0)},0\right)\right|\\
&& \displaystyle \times
\exp \int_0^1 \left(\frac{1}{a_2(x,0)}
b_2^0\left(t-\int_x^1\frac{d \xi}{a_2(\xi,0)},x,0\right)-\frac{1}{a_1(x,0)}
b_1^0\left(t-\int_x^1\frac{d \xi}{a_1(\xi,0)},x,0\right)\right)dx
\end{eqnarray*}
and  replace the assumption \reff{Fred} by the weaker assumption
\beq
\label{Fred1}
R(t) \not= 1 \mbox{ for all } t \mbox{ or } S(t) \not= 1 \mbox{ for all } t.
\ee
Write
\beq
\label{infnorm}
\|u\|_0:=\max_{1 \le j \le n} \max_{t,x} |u_j(t,x)|.
\ee

The following theorem is the main result of this section.

\begin{thm}
\label{thm:hopf}
Suppose \reff{hyp0} and \reff{hyp}. Further, suppose \reff{Fred} if $n\ge 3$ and \reff{Fred1} if  $m=1,n=2$.
Finally, suppose that  the linearized problem
\beq\label{evp}
\left.
\begin{array}{rl}
\displaystyle
\partial_tu_j  + a_j(x,0)\partial_xu_j  + \sum_{k=1}^n\d_{u_k}b_j(t,x,0,u_0(t,x))u_k  = 0, &j=1,\dots,n,\\
\displaystyle
u_j(t+2\pi,x)=u_j(t,x), &j=1,\dots,n,\\
\displaystyle
u_j(t,0) = \sum\limits_{k=m+1}^nr_{jk}(t,0)u_k(t,0),& j=1,\ldots,m,\\
\displaystyle
u_j(t,1) = \sum\limits_{k=1}^mr_{jk}(t,0)u_k(t,1), & j=m+1,\ldots,n
\end{array}
\right\}
\ee
does not have a solution $u \not= 0$.

Then there exist $\eps>0$, $\delta >0$ and a $C^\infty$-map
$\hat u : \R \times [0,1] \times[-\eps,\eps] \to \R^n$  such that the following is true:

(i) For all $\la\in [-\eps,\eps]$ the function $\hat u(\cdot,\cdot,\la)$
is a solution  to (\ref{eq:1.1})--(\ref{eq:1.2}).

(ii) If $u$ is a solution to (\ref{eq:1.1})--(\ref{eq:1.2}) with 
$ |\la|+ \|u-u_0\|_{0}<\de$,
then
$u(t,x)=\hat u(t,x,\la)$ for all  $t\in\R$ and $x\in[0,1]$.
In particular, $u_0=\hat u(\cdot,\cdot,0)$ is  $C^\infty$-smooth.
\end{thm}

\begin{rem}\rm
\label{mn1}
Consider   (\ref{eq:1.1})--(\ref{eq:1.2}) with $m=n=1$, 
$a_1(x,\la)=1+\la^2$ and $b_1(t,x,\la,u)=b(t,x)$. 
Then   (\ref{eq:1.1})--(\ref{eq:1.2}) reads 
$$
\d_tu+(1+\la^2)\d_xu+b(t,x)=0,\; u(t+2\pi,x)=u(t,x), \; u(t,0)=0.
$$ 
The solution is explicitely given by the formula
$$
u(t,x)=-\frac{1}{1+\la^2}\int_0^xb\left(\frac{\xi-x}{1+\la^2}+t,\xi\right) d\xi.
$$
It follows that $u$ depends smoothly on $\la$ if and only if $b$ is smooth in $t$. In other words, solutions to   (\ref{eq:1.1})--(\ref{eq:1.2})
do not depend pointwise smoothly on $\la$ if the functions $b_j(\cdot,x,\la,u)$ are not smooth, in general.
\end{rem}

\begin{rem}\rm
\label{rema}
Consider   (\ref{eq:1.1})--(\ref{eq:1.2}) with $m=1, n=2$, $r_{12}(t,\la)=r_{21}(t,\la)=1$, $a_1(x,\la)=-a_2(x,\la)=1+\la^2$ and $b_j(t,x,\la,u)=0$, i.e.
$$
\d_tu_1+(1+\la^2)\d_xu_1=\d_tu_2-(1+\la^2)\d_xu_2=0,\; u_1(t,0)-u_2(t,0)= u_1(t,1)-u_2(t,1)=0.
$$
Integration along characteristic curves combined with the boundary conditions gives (cf. also (\ref{rep1})--(\ref{rep2}))
\begin{eqnarray}
&&u_1(t,x)=u_2(t-(1+\la^2) x,0), \label{1a}\\
&&u_2(t,x)=u_1(t+(1+\la^2)(x-1),1). \label{2a}
\end{eqnarray}
Inserting \reff{2a} into \reff{1a} and putting $x=1$, we get
\beq
\label{3}
u_1(t,1)=u_1(t-2(1+\la^2),1).
\ee
If $(1+\la^2)/2\pi$ is irrational, then the functional equation \reff{3} does not have nontrivial continuous solutions. If 
$$
\frac{1+\la^2}{2\pi}=\frac{p}{q} \mbox{ with } p \in \Z \mbox{ and } q \in \N,
$$
then any $2\pi/q$-periodic  function is a solution to \reff{3}, and hence there are infinitely many solutions.
This means that  in this case all assumptions of Theorem \ref{thm:hopf} (with $u_0=0$) are fulfilled with the exception of  \reff{Fred1} 
(because of $R_0(t)=S_0(t)=1$ for all $t$), but the conclusion of  Theorem \ref{thm:hopf} is wrong. 
\end{rem}

\subsection{Proof of Theorem \ref{thm:hopf}}
\label{Proof3.1}

In this subsection we prove Theorem \ref{thm:hopf}. Hence, we suppose all assumptions of this theorem to be fulfilled.
In order to use  Theorem \ref{thm:IFT},
we are going to represent problem  (\ref{eq:1.1})--(\ref{eq:1.2}) 
in the setting of Section \ref{IFT}.
Let us introduce the corresponding objects:

We denote by $U_0$ the space of all continuous maps $u:\R \times [0,1] \to \R^n$ such that $u(\cdot,x)$ is $2\pi$-periodic,
with the norm \reff{infnorm}. The $C_0$-group $T(s)$ on $U_0$ is defined by
$$
\left(T(s)u\right)(t,x):=u(t+s,x).
$$
Hence, the corresponding infinitesimal generator is
$
A:=\d_t,
$
and the domain of definition of $A^l$ is
$$
U_l:=\{u \in U_0:\; \d^j_tu \in U_0 \mbox{ for all } j=0,1,\ldots,l\}.
$$
The spaces $U_l$  are Banach spaces with the norms
$$
\|u\|_l:=\sum_{j=0}^l\|\d_t^ju\|_0.
$$
Let  us introduce characteristics of the system \reff{eq:1.1}
\beq
\label{charom}
\tau_j(t,x,\xi,\la):=t+\int_x^\xi \frac{d\eta}{a_j(\eta,\la)}.
\ee
For  $t \in \R$, $x,\xi \in [0,1]$, $\la \in [-\eps_0,\eps_0]$ and $u,v \in U_0$ we define 
\beq
\label{cdef}
c_j(t,x,\xi,\la,v):=\exp \int_x^\xi
\frac{\partial_{u_j}b_j(\tau_j(t,x,\eta,\la),\eta,\la,v(\tau_j(t,x,\eta,\la),\eta))}{a_{j}(\eta,\la)}\,d\eta.
\ee
Further,  for  $t \in \R$, $x \in [0,1]$, $\la \in [-\eps_0,\eps_0]$ and $u,v \in \R^n$ we define 
\beq
\label{fdef}
f_j(t,x,\la,u,v):=\partial_{u_j}b_j(t,x,\la,v)u_j-b_j(t,x,\la,u).
\ee

The following calculations, the so-called integration along the characteristics of the hyperbolic system \reff{eq:1.1}, are well-known:
Let $u$ be a solution to \reff{eq:1.1}--\reff{eq:1.2} and  $v \in U_0$ be arbitrary fixed. Then
$$
\partial_tu_j(t,x)  + a_j(x,\la)\partial_xu_j(t,x) +\d_{u_j}b_j(t,x,\la,v(t,x))u_j(t,x)  = f_j(t,x,\la,u(t,x),v(t,x)),
$$
and hence
\begin{eqnarray*}
&&\frac{d}{d\xi}u_j(\tau_j(t,x,\xi,\la),\xi)+\frac{\d_{u_j}b_j(\tau_j(t,x,\xi,\la),\xi,\la,v(\tau_j(t,x,\xi,\la),\xi))}{a_j(\xi,\la)}u_j(\tau_j(t,x,\xi,\la),\xi)\\
&&=\frac{f_j(\tau_j(t,x,\xi,\la),\xi,\la,u(\tau_j(t,x,\xi,\la),\xi),v(\tau_j(t,x,\xi,\la),\xi))
}{a_j(\xi,\la)}.
\end{eqnarray*}
Integrating this in $\xi$ from zero to $x$ for $j=1,\ldots,m$ and  from $x$ to one for $j=m+1,\ldots,n$, and, using the boundary conditions \reff{eq:1.2},
we get
\begin{eqnarray}
\label{rep1}
\lefteqn{
u_j(t,x)=c_j(t,x,0,\la,v)\sum_{k=m+1}^nr_{jk}(\tau_j(t,x,0,\la),\la)u_k(\tau_j(t,x,0,\la),0)}\nonumber\\
&&+\int_0^x \frac{c_j(t,x,\xi,\la,v)}{a_j(\xi,\la)}f_j(\tau_j(t,x,\xi,\la),\xi,\lambda,u(\tau_j(t,x,\xi,\lambda),\xi),v(\tau_j(t,x,\xi,\lambda),\xi))d\xi,\nonumber\\
&&j=1,\ldots,m,
\end{eqnarray}
\begin{eqnarray}
\label{rep2}
\lefteqn{
u_j(t,x)=c_j(t,x,1,\la,v)\sum_{k=1}^m r_{jk}(\tau_j(t,x,1,\la),\la)u_k(\tau_j(t,x,1,\la),1)}\nonumber\\
&&-\int_x^1 \frac{c_j(t,x,\xi,\la,v)}{a_j(\xi,\la)}f_j(\tau_j(t,x,\xi,\la),\xi,\la,u(\tau_j(t,x,\xi,\la),\xi),v(\tau_j(t,x,\xi,\la),\xi))d\xi,\nonumber\\
&&j=m+1,\ldots,n.
\end{eqnarray}
And vice versa: If $v\in U_1$ is given, then any  solution $u\in U_1$ to \reff{rep1}-\reff{rep2}  
is a solution to \reff{eq:1.1}--\reff{eq:1.2} (in particular, $u$ is differentiable not only with respect to $t$, but also with respect to $x$). 
Remark that the right-hand sides of  \reff{rep1}-\reff{rep2} depend on the artificially introduced function $v$, but the solutions $u$ to 
\reff{rep1}-\reff{rep2} do not.

If we set $b(t,x,\la,u):=(b_1(t,x,\la,u),\ldots,b_n(t,x,\la,u))$ and  
\beq
\label{fdef1}
f(t,x,\la,u,v):=(f_1(t,x,\la,u,v),\ldots,f_n(t,x,\la,u,v)),
\ee
then the nonlinear map $f(t,x,\la,\cdot,v)$ is the difference of   the diagonal part of  the linear map
$\d_ub(t,x,\la,u,v)|_{u=v}$ and of the
nonlinear map $b(t,x,\la,\cdot)$. Hence,
the diagonal part of $\d_uf(t,x,\la,u,v)|_{u=v}$ vanishes, i.e. 
$
\partial_{u_j}f_j(t,x,\la,u,v)|_{u=v}=0,
$
cf. \reff{fdef}. This will be used later on in the proof of the key Lemma \ref{B4anw}.

For $\la \in[-\eps_0,\eps_0]$ and $v \in U_0$ we define linear bounded operators $C(\la,v): U_0 \to U_0$ by
\beq
\label{Cdef}
\left(C(\la,v)u\right)_j(t,x):=
\left\{
\begin{array}{l}
\displaystyle
c_j(t,x,0,\la,v)\sum_{k=m+1}^nr_{jk}(\tau_j(t,x,0,\la),\la)u_k(\tau_j(t,x,0,\la),0) \\\,\;\;\;\;\;\;\;\;\;\;\;\;\;\;\;\;\;\;\;\;\;\;\;\;\;\;\;\;\;\;\;\;\;\;\;\;\;\;\;\;\;\;\;\;\;
\mbox{ for }j=1,\ldots,m,\\
\displaystyle
c_j(t,x,1,\la,v)\sum_{k=1}^mr_{jk}(\tau_j(t,x,1,\la),\la)u_k(\tau_j(t,x,1,\la),1) \\\,\;\;\;\;\;\;\;\;\;\;\;\;\;\;\;\;\;\;\;\;\;\;\;\;\;\;\;\;\;\;\;\;\;\;\;\;\;\;\;\;\;\;\;\;\;
\mbox{ for }
j=m+1,\ldots,n.
\end{array}
\right.
\ee
Further, for $\la \in[-\eps_0,\eps_0]$ we define nonlinear operators  $D(\la,\cdot,\cdot): U_0\times U_0 \to U_0$ by
\begin{eqnarray}
\label{Fdef}
&&D(\la,u,v)_j(t,x):=\nonumber\\
&&
\int_{x_j}^x \frac{c_j(t,x,\xi,\la,v)}{a_j(\xi,\la)}f_j(\tau_j(t,x,\xi,\lambda),\xi,\lambda,u(\tau_j(t,x,\xi,\lambda),\xi),v(\tau_j(t,x,\xi,\lambda),\xi))d\xi.
\end{eqnarray}
Here and in what follows we write
\beq
\label{xj}
x_j:=\left\{
\begin{array}{rcl}
0 & \mbox{for} & j=1,\dots,m,\\
1 & \mbox{for} & j=m+1,\dots,n.
\end{array}
\right.
\ee
In this notation the system \reff{rep1}--\reff{rep2} can be written as  the operator equation
\beq
\label{abstract}
u=C(\la,v)u+D(\la,u,v).
\ee
Remark again that the right-hand side of  \reff{abstract} depends on the artificially introduced function $v$, but the solutions $u$ to 
\reff{abstract} do not. In particular, we have
\beq
\label{dv}
\d_vC(0,v)(u_0,w)+\d_vD(\la,u_0,v)w=0 \mbox{ for all } v,w \in U_0.
\ee

\begin{lem}
\label{prop}
The maps $(s,u,v) \in \R \times U_0^2 \mapsto T(s)C(\la,T(-s)v)T(-s)u \in U_0$ and
$(s,u,v) \in \R \times U_0^2 \mapsto T(s)D(\la,T(-s)u,T(-s)v) \in U_0$ are $C^\infty$-smooth for all $\la \in  [-\eps_0,\eps_0]$.
\end{lem}
\begin{proof} From \reff{charom} we have $\tau_j(t+s,x,\xi,\la)=\tau_j(t,x,\xi,\la)+s$.
Then \reff{cdef}, \reff{Cdef} and  \reff{Fdef} yield
\begin{eqnarray*}
&&\left(T(s)C(\la,T(-s)v)T(-s)u\right)_j(t,x)=\\
&&c^s_j(t,x,0,\la,v)\sum_{k=m+1}^nr_{jk}(\tau_j(t,x,0,\la)+s,\la)u_k(\tau_j(t,x,0,\la),0) 
\mbox{ for }j=1,\ldots,m,
\end{eqnarray*}
\begin{eqnarray*}
&&\left(T(s)C(\la,T(-s)v)T(-s)u\right)_j(t,x)=\\
&&c^s_j(t,x,1,\la,v)\sum_{k=1}^mr_{jk}(\tau_j(t,x,1,\la)+s,\la)u_k(\tau_j(t,x,1,\la),1) 
\mbox{ for }
j=m+1,\ldots,n
\end{eqnarray*}
and
\begin{eqnarray*}
&&T(s)D(\la,T(-s)u,T(-s)v)_j(t,x)=\nonumber\\
&&
\int_{x_j}^x \frac{c^s_j(t,x,\xi,\la,v)}{a_j(\xi,\la)}f_j(\tau_j(t,x,\xi,\lambda)+s,\xi,\lambda,u(\tau_j(t,x,\xi,\lambda),\xi),v(\tau_j(t,x,\xi,\lambda),\xi))d\xi
\end{eqnarray*}
with
$$
c^s_j(t,x,\xi,\la,v):=\exp \int_x^\xi
\frac{\partial_{u_j}b_j(\tau_j(t,x,\eta,\la)+s,\eta,\la,v(\tau_j(t,x,\eta,\la),\eta))}{a_{j}(\eta,\la)}\,d\eta.
$$
Since the functions $b_j$ are smooth,
well-known differentiability properties of Nemytskii operators (see, e.g., \cite[Lemma 2.4.18]{Abraham}, 
\cite[Lemma 6.1]{Amann}) yield the claim of the lemma.
\end{proof}

\begin{rem}\rm
\label{nonsm}
Suppose the coefficients $a_j$ to be  independent on $\la$. Then also the characteristics $\tau_j$  are independent on $\la$, and hence  the  maps 
$\la \mapsto C(\la,v) \in  {\cal L}(U_0)$ and $\la\mapsto D(\la,u,v) \in  U_0$ 
are smooth for all $u,v \in U_0$. Therefore, in this case for $\la \approx 0$ the equation \reff{abstract} can be solved with respect to $u \approx u_0$ 
by means of the classical implicit function theorem under natural assumptions.

But,  if  the coefficients $a_j$ depend on $\la$, then, unfortunately, the maps $\la 
\mapsto C(\la,v)u \in  U_0$ and $\la\mapsto D(\la,u,v) \in  U_0$ 
are not smooth
(if $u$ and $v$ are only continuous and not smooth),  in general.
This makes the question, if the data-to-solution map corresponding to \reff{abstract} 
is smooth, very delicate in the case if  the coefficients $a_j$ depend on $\la$.
\end{rem}

\begin{rem}\rm
\label{tabh}
If the  coefficients $a_j$ would depend on $t$, then the characteristics $\tau_j$ are defined as the solutions to the initial value problem
$$
\d_\xi\tau_j(t,x,\xi,\la)=\frac{1}{a_j(\tau_j(t,x,\xi,\la),\xi,\la)},\; \tau_j(t,x,x,\la)=t.
$$
But then $\tau_j(t+s,x,\xi,\la)\not=\tau_j(t,x,\xi,\la)+s$, in general, and the proof of  Lemma \ref{prop} could not be generalized to this case, in general.
To get a result similar to Theorem~\ref{thm:hopf}, but with $t$-dependent  coefficients $a_j$,
 it turns out that additional conditions of the type \reff{Fred} should be assumed. The number of those conditions is $k+1$ in order  the data-to-solution map 
to be $C^k$-smooth (see \cite{kmit}).
\end{rem}

In what follows we will use the operators $\CC(\la), \D(\la) \in \LL(U_0)$ defined for $\la \in   [-\eps_0,\eps_0]$ by
\beq
\label{CDdef}
\CC(\la):=C(\la,u_0),\; \D(\la):=\d_uD(\la,u,u_0)|_{u=u_0}.
\ee
From   \reff{Cdef} and \reff{Fdef} it follows that
\beq
\label{CCdef}
\left(\CC(\la)u\right)_j(t,x)=
\left\{
\begin{array}{l}
\displaystyle
\sum_{k=m+1}^nc_{jk}(t,x,\la)u_k(\tau_j(t,x,0,\la),0)
\mbox{ for }j=1,\ldots,m,\\
\displaystyle
\sum_{k=1}^mc_{jk}(t,x,\la)u_k(\tau_j(t,x,1,\la),1) 
\mbox{ for }
j=m+1,\ldots,n
\end{array}
\right.
\ee
with
\beq
\label{tildec}
c_{jk}(t,x,\la):=c_j(t,x,x_j,\la,u_0)r_{jk}(\tau_j(t,x,x_j,\la),\la)
\ee
and
\beq
\label{dDdef}
\left[\D(\la)u\right]_j(t,x)=\sum_{k\not=j}\int_{x_j}^xd_{jk}(t,x,\xi,\la)u_k(\tau_j(t,x,\xi,\lambda),\xi)d\xi
\ee
with
\beq
\label{tilded}
d_{jk}(t,x,\xi,\la):=
-\frac{c_j(t,x,\xi,\la,u_0)}{a_j(\xi,\la)}\d_{u_k}b_j(\tau_j(t,x,\xi,\lambda),\xi,\lambda,u_0(\tau_j(t,x,\xi,\lambda),\xi)).
\ee

\begin{lem}
\label{prop1}
Suppose \reff{Fred} if $n\ge 3$ and  \reff{Fred1} if $m=1,n=2$.
Then there exist $\eps_1 \in (0,\eps_0]$ and $d>0$ such that for all $\la \in  [-\eps_1,\eps_1]$ the operator
$I-\CC(\la)$ is bijective from $U_0$ to $U_0$ and $\left\|(I-\CC(\la))^{-1}\right\|_{\LL(U_0)} \le d$.
\end{lem}
\begin{proof}
We will follow the ideas of  \cite[Lemma 3.1]{KR2} and  \cite[Lemma 2.2 (i)]{KR3}.

Denote by $U_0(m)$ the space of all continuous maps 
$v: \R \times [0,1] \to \R^m$ with $v(t+2\pi,x)=v(t,x)$ for all  $t \in \R$ and $x \in [0,1]$, with the norm
$$
\|v\|:=\max_{1 \le j \le m} \max_{t,x} |v_j(t,x)|.
$$
Similarly we define the space $U_0(n-m)$. The spaces $U_0$ and $U_0(m) \times 
U_0(n-m)$ will be identified, i.e. elements $u \in U_0$ will be written as $u=(v,w)$ with $v \in U_0(m)$ and $w \in U_0(n-m)$.
Then the operators $\CC(\la)$ work as
\beq
\label{KL}
\CC(\la)u=(K(\la)w,L(\la)v) \mbox{ for } u=(v,w),
\ee
where the linear bounded operators $K(\la):U_0(n-m) \to U_0(m)$ and  $L(\la):U_0(m) \to U_0(n-m)$
are defined by the right hand side of \reff{Cdef}.

Let $f=(g,h) \in U_0$ with $g \in U_0(m)$ and $h \in U_0(n-m)$ be arbitrarily given.
We have $u=\CC(\la)u+f$ if and only if
$v=K(\la)w+g, \; w=L(\la)v+h,$
i.e., if and only if
\beq
\label{inserted}
v=K(\la)(L(\la)v+h)+g, \; w=L(\la)v+h.
\ee
This holds if and only if 
$$v=K(\la)w+g, \; w=L(\la)(K(\la)w+g)+h.$$

If there exist $\eps_1 \in (0,\eps_0]$ and $d >0$ such for all $ \la \in [-\eps_1,\eps_1]$ the operator $I-K(\la)L(\la)$ is invertible on $U_0(m)$ and
\beq
\label{Prod1}
\|(I-K(\la)L(\la))^{-1}\|_{\LL(U_0(m))}\le d \mbox{ for all } \la \in [-\eps_1,\eps_1],
\ee
then \reff{inserted} is uniquely solvable with respect to $(v,w)$ and $\|(v,w)\|_0 \le \mbox{ const } \|(f,g)\|_0$,
where the constant does not depend on $\la, f$ and $g$. The conclusion of the lemma is true. 
Similarly, if there exist $\eps_1 \in (0,\eps_0]$ and $d >0$ such for all $ \la \in [-\eps_1,\eps_1]$ the operator $I-L(\la)K(\la)$ is invertible on $U_0(n-m)$ and
\beq
\label{Prod1a}
\|(I-L(\la)K(\la))^{-1}\|_{\LL(U_0(n-m))}\le d \mbox{ for all } \la \in [-\eps_1,\eps_1],
\ee
then the conclusion of the lemma is true also.

Let us prove \reff{Prod1} and \reff{Prod1a} for $n\ge 3$.
Because of \reff{CCdef} the equation $v=K(\la)L(\la)v+\tilde g$ is equivalent to
\begin{eqnarray}
\label{veq}
&&\sum_{k=m+1}^nc_{jk}(t,x,\la)\sum_{l=1}^mc_{kl}(\tau_j(t,x,0,\la),0,\la)v_l(\tau_k(\tau_j(t,x,0,\la),0,1,\la),1)\nonumber\\
&&=v_j(t,x)-\tilde g_j(t,x), \; j=1,\ldots,m.
\end{eqnarray}
At $x=1$ it reads
\begin{eqnarray}
\label{veqx}
&&\sum_{k=m+1}^nc_{jk}(t,1,\la)\sum_{l=1}^mc_{kl}(\tau_j(t,1,0,\la),0,\la)v_l(\tau_k(\tau_j(t,1,0,\la),0,1,\la),1)\nonumber\\
&&=v_j(t,1)-\tilde g_j(t,1), \; j=1,\ldots,m.
\end{eqnarray}
Because of \reff{b0def}, \reff{cdef}, \reff{xj} and \reff{tildec} the coefficients in \reff{veqx} are
\begin{eqnarray*}
&&c_{jk}(t,1,\la)c_{kl}(\tau_j(t,1,0,\la),0,\la)\nonumber\\
&&=r_{jk}(\tau_j(t,1,0,\la),\la)r_{kl}(\tau_k(\tau_j(t,1,0,\la),0,1,\la),\la)\times\\
&& \times\exp\int_0^1\left(
\frac{b_k^0(\tau_k(\tau_j(t,1,0,\la),0,\eta,\la),\eta,\la)}{a_k(\eta,\la)}-\frac{b_j^0(\tau_j(t,1,\eta,\la),\eta,\la)}{a_j(\eta,\la)}\right)d\eta.
\end{eqnarray*}
Hence, for $\|v\| \le 1$ the absolute value of the right-hand side of \reff{veqx} can be estimated from above, for example, by
$$
R(\la):=\max_{1 \le j \le m} \sum_{k=m+1}^n\sum_{l=1}^m \max_{t,s}|r_{jk}(t,\la)r_{kl}(s,\la)|\exp\int_0^1 \max_{t,s}\left(\frac{b_k^0(t,x,\la)}{a_k(x,\la)}
-\frac{b_j^0(s,x,\la)}{a_j(x,\la)}\right)dx.
$$ 
Similarly, the equation $w=L(\la)K(\la)w+\tilde h$ is equivalent to
\begin{eqnarray*}
&&\sum_{k=1}^mc_{jk}(t,x,\la)\sum_{l=m+1}^nc_{kl}(\tau_j(t,x,1,\la),1,\la)w_l(\tau_k(\tau_j(t,x,1,\la),1,0,\la),0)\nonumber\\
&&=w_j(t,x)-\tilde h_j(t,x), \; j=m+1,\ldots,n.
\end{eqnarray*}
At $x=0$ it reads
\begin{eqnarray}
\label{veqxa}
&&\sum_{k=1}^mc_{jk}(t,0,\la)\sum_{l=m+1}^nc_{kl}(\tau_j(t,0,1,\la),1,\la)w_l(\tau_k(\tau_j(t,0,1,\la),1,0,\la),0)\nonumber\\
&&=w_j(t,0)-\tilde h_j(t,0), \; j=m+1,\ldots,n.
\end{eqnarray}
The coefficients in \reff{veqxa} are
\begin{eqnarray*}
&&c_{jk}(t,0,\la)c_{kl}(\tau_j(t,0,1,\la),1,\la)\nonumber\\
&&=r_{jk}(\tau_j(t,0,1,\la),\la)r_{kl}(\tau_k(\tau_j(t,0,1,\la),1,0,\la),\la)\times\\
&& \times\exp\int_0^1\left(\frac{b_j^0(\tau_j(t,0,\eta,\la),\eta,\la)}{a_j(\eta,\la)}-
\frac{b_k^0(\tau_k(\tau_j(t,0,1,\la),1,\eta,\la),\eta,\la)}{a_k(\eta,\la)}\right)d\eta.
\end{eqnarray*}
Hence, for $\|w\| \le 1$ the absolute value of the right-hand side of \reff{veqxa} can be estimated from above, for example, by
$$
S(\la):=\max_{m+1 \le j \le n} \sum_{k=1}^m\sum_{l=m+1}^n \max_{t,s}|r_{jk}(t,\la)r_{kl}(s,\la)|\exp\int_0^1 \max_{t,s}\left(\frac{b_j^0(t,x,\la)}{a_j(x,\la)}
-\frac{b_k^0(s,x,\la)}{a_k(x,\la)}\right)dx.
$$ 
Since $R(\la)$ and $S(\la)$ depend continuously on $\la$ and  $R(0)<1$ or  $S(0)<1$ (see  \reff{Fred}),
 there exist $\eps_1 \in (0,\eps_0]$ and positive  $d<1$ such for all $\la \in [-\eps_1,\eps_1]$ we have $R(\la) \le d$ or  $S(\la) \le d$.

Suppose   $R(\la) \le d$. Then the system \reff{veqx} has a unique solution $(v_1(t,1),\ldots,v_m(t,1))$, which is continuous and $2\pi$-periodic in $t$, and it holds
$$
\max_{1 \le j \le m} \max_{t}|v_j(t,1)| \le \mbox{ const } \|\tilde g\|_0,
$$
where the constant does not depend on $\la$ and $\tilde g$. Inserting this into \reff{veq}, we see that  for all $\la \in [-\eps_1,\eps_1]$ 
the system \reff{veq} has a unique solution $v$, and it holds $\|v\|_0 \le \mbox{const } \|\tilde g\|_0$,
where the constant does not depend on $\la$ and $\tilde g$.
Hence, \reff{Prod1} is true. Similary one shows that  from $S(\la) \le d$ it follows  \reff{Prod1a}.

Now, let us consider the case $m=1, n=2$. In this case the equation \reff{veqx} reads
\begin{eqnarray*}
&&c_{12}(t,1,\la)c_{21}\left(t-\int_0^1\frac{dx}{a_1(x,\la)},0,\la\right)v_1\left(t+\int_0^1\left(\frac{1}{a_2(x,\la)}-\frac{1}{a_1(x,\la)}\right)dx,1\right)\nonumber\\
&&=v_1(t,1)-\tilde g_1(t,1). 
\end{eqnarray*}
If
\beq\label{firstcond}
\left|c_{12}(t,1,0)c_{21}\left(t-\int_0^1\frac{dx}{a_1(x,0)},0,0\right)\right|\ne 1 \mbox{ for all } t,
\ee
then there exists $\eps_1>0$ such that for all  $\la \in [-\eps_1,\eps_1]$ the first equation has a unique solution $v_1(t,1)$, which is 
continuous and $2\pi$-periodic with respect to $t$, and it holds
$$
\max_{t}|v_j(t,1)| \le \mbox{const} \max_{t} |\tilde g_1(t,1)|,
$$
where the constant does not depend on $\la$ and $\tilde g_1$. Then, as above,  the conclusion of the lemma is true.
On the other hand, equation \reff{veqxa} is
\begin{eqnarray*}
&&c_{21}(t,0,\la)c_{12}\left(t+\int_0^1\frac{dx}{a_2(x,\la)},1,\la\right)w_2\left(t+\int_0^1\left(\frac{1}{a_2(x,\la)}-\frac{1}{a_1(x,\la)}\right)dx,0\right)\nonumber\\
&&=w_1(t,0)-h_2(t,0), 
\end{eqnarray*}
and, similarly to the above, one shows that  the conclusion of the lemma is true if
\beq
\label{seccond}
\left|c_{21}(t,0,0)c_{12}\left(t+\int_0^1\frac{dx}{a_2(x,0)},1,0\right)\right|\not=1 \mbox{ for all } t.
\ee
But \reff{charom}, \reff{cdef} and \reff{tildec} yield that the left-hand side of \reff{firstcond} is $R_0\left(t-\int_0^1\frac{dx}{a_1(x,0)}\right)$ and the  left-hand side of \reff{seccond} is $S_0\left(t+\int_0^1\frac{dx}{a_2(x,0)}\right)$.
Hence, assumption \reff{Fred1} implies one of the conditions  \reff{firstcond} and  \reff{seccond}.
\end{proof}

\begin{rem}\rm
\label{rema1}
Let us revisit the example of Remark \ref{rema}, i.e., consider   (\ref{eq:1.1})--(\ref{eq:1.2}) with $m=1, n=2$, $r_{12}(t,\la)=r_{21}(t,\la)=1$, $a_1(x,\la)=-a_2(x,\la)=1+\la^2$, $b_j(x,t,\la)=0$
and $u_0=0$.
Then   
$$
\left(\CC(\la)u\right)_j(t,x)=
\left\{
\begin{array}{l}
\displaystyle
u_2(0,t-(1+\la^2) x)
\mbox{ for }j=1,\\
\displaystyle
u_1(1,t+(1+\la^2)(x-1))
\mbox{ for }
j=2.
\end{array}
\right.
$$
Hence $I-\CC(\la)$ is bijective on $U_0$ if and only if  $(1+\la^2)/2\pi$ is irrational. This shows again that the map $\la \in \R \mapsto \CC(\la) \in {\cal L}(U_0)$ is not continuous.
\end{rem}

In what follows we will work with the Banach space $V$ and the Hilbert space $H$ defined as follows:
\begin{eqnarray*}
&&V:=\{u \in U_0: \; \d_tu,\d_xu \in U_0\},\; \|u\|_V:= \|u\|_0+ \|\d_tu\|_0+ \|\d_xu\|_0,\\
&&H:=L^2\left((0,1) \times (0,2\pi);\R^n\right), \; \langle u,v \rangle := \int_0^{2\pi}\int_0^1u(t,x)v(x,t)dxdt.
\end{eqnarray*}
The space $V$ is densely and compactly embedded into $U_0$,  while $U_0$ is continuously embedded into $H$.

\begin{lem}
\label{B3anw}
It holds
$$
\lim_{\la \to 0} \sup_{\|u\|_0\le 1}\langle \left(\CC(\la)+\D(\la)-\CC(0)-\D(0)\right)u,h\rangle=0 \mbox{ for all } h \in H.
$$
\end{lem}
\begin{proof}
Take $\la \in [-\eps_0,\eps_0]$, $u \in U$ and $h \in H$. Accordingly to \reff{CCdef},
 the scalar product $\left\langle \left(\CC(\la)-\CC(0)\right)u,h\right\rangle$ equals 
$$
\sum_{j,k=1}^n\int_0^1\int_0^{2 \pi}\left(c_{jk}(t,x,\la)u_k(\tau_j(t,x,x_j,\la),x_j)-c_{jk}(t,x,0)u_k(\tau_j(t,x,x_j,0),x_j)\right)h_j(x,t) dt dx.
$$
Denote by $t_j(\cdot,x,\xi,\la)$ the inverse function of $\tau_j(\cdot,x,\xi,\la)$ (cf. \reff{charom}), i.e.,
$$
t_j(\tau,x,\xi,\la):=\tau-\int_x^\xi\frac{d\eta}{a_j(\eta,\la)}.
$$
Then the change of integration variables $t \mapsto s$, which is defined by
$$
t=\om_j(s,x,\la):=t_j(\tau_j(s,x,x_j,0),x,x_j,\la)=s+\int_x^{x_j}\left(\frac{1}{a_j(\eta,\la)}-\frac{1}{a_j(\eta,0)}\right)d\eta,
$$
yields
$$
\begin{array}{ll}
\displaystyle
\int_0^{2\pi}\left(c_{jk}(t,x,\la)u_k(\tau_j(t,x,x_j,\la),x_j)-c_{jk}(t,x,0)u_k(\tau_j(t,x,x_j,0),x_j)\right)h_j(t,x) dt\nonumber \\
\displaystyle =\int_0^{2 \pi}\left(c_{jk}(\om_j(s,x,\la),x,\la)h_j(\om_j(s,x,\la),x)-c_{jk}(s,x,0)h_j(s,x))\right)u_k(\tau_j(s,x,x_j,0),x_j)ds.\nonumber
\end{array}
$$
Hence, 
\begin{eqnarray*}
&&\sup_{\|u\|_0\le 1}\left|\left\langle \left(\CC(\la)-\CC(0)\right)u,h\right\rangle\right|\\
&&\le\sum_{j,k=1}^n\int_0^1\int_0^{2 \pi}\left|\left(c_{jk}(\om_j(s,x,\la),x,\la)h_j(\om_j(x,s,\la),x)-c_{jk}(s,x,0)h_j(s,x))\right)\right| ds dx \to 0
\end{eqnarray*}
for $\la \to 0$.
Here we used $\om_j(s,x,0)=s$ and the continuity in the mean of $L^2$-functions.

Similarly, accordingly to \reff{dDdef}, we have 
$$
\left\langle \D(\la)u,h\right\rangle=\sum_{j=1}^n \sum_{k\not=j}^n \int_0^1\int_{x_j}^x\int_0^{2 \pi} 
d_{jk}(t,x,\xi,\la)u_k(\tau_j(t,x,\xi,\la),\xi)h_j(x,t) dt d\xi dx.
$$
Hence, using the change of the integration variable
$
t=t_j(\tau_j(s,x,\xi,0),x,\xi,\la),
$
we get 
$$
\lim_{\la \to 0}\sup_{\|u\|_0\le 1}\left\langle \left(\D(\la)-\D(0)\right)u,h\right\rangle=0.
$$
\end{proof}

\begin{lem}
\label{B4anw}
There exists $d>0$ such that for all $\la \in [-\eps_0,\eps_0]$ and $v \in V$ it holds
$\D(\la)^2v \in V$, $\D(\la)\CC(\la)v \in V$ and  
$$
\|\D(\la)^2v\|_V + \|\D(\la)\CC(\la)v\|_V \le d\|v\|_{0}.
$$
\end{lem}
\begin{proof}
We have to show that for all $\la \in [-\eps_0,\eps_0]$ and $v \in V$  it holds $\d_t(\D(\la)\CC(\la)v) \in U_0$,
$\d_x(\D(\la)\CC(\la)v) \in U_0$,
$\d_t(\D(\la)^2v) \in U_0$, 
$\d_x(\D(\la)^2v) \in U_0$
and
$$
\|\d_t(\D(\la)\CC(\la)v)\|_0 +\|\d_x(\D(\la)\CC(\la)v)\|_0
+\|\d_t(\D(\la)^2v)\|_0 +\|\d_x(\D(\la)^2v)\|_0
\le \mbox{const}\|v\|_0.
$$
From \reff{CCdef} and \reff{dDdef} it follows that for all $v \in V$ we have
\begin{eqnarray*}
\d_t(\CC(\la)v)&=&C^t_0(\la)v+\CC(\la)\d_tv,\\
\d_x(\CC(\la)v)&=&C^x_0(\la)v+C^x_1(\la)\d_tv, \\ 
\d_t(\D(\la)v)&=&D^t_0(\la)v+\D(\la)\d_tv,\\
\d_x(\D(\la)v)&=&D^x_0(\la)v+D^x_1(\la)\d_tv
\end{eqnarray*}
with the operators $C^t_0(\la), C^t_1(\la), C^x_0(\la),  C^x_1(\la), D^t_0(\la), D^t_1(\la), D^x_0(\la),  D^x_1(\la) \in \LL(U_0)$ defined by
\begin{eqnarray*}
\left[C^t_0(\la)v\right]_j(t,x)&:=&\sum_{k=1}^n\d_tc_{jk}(t,x,\la)
v_k(\tau_j(t,x,x_j,\la),x_j),\\
\left[C^x_0(\la)v\right]_j(t,x)&:=&\sum_{k=1}^n\d_xc_{jk}(t,x,\la)
v_k(\tau_j(t,x,x_j,\la),x_j),\\
\left[C^x_1(\la)w\right]_j(t,x)&:=&-\sum_{k=1}^n\frac{c_{jk}(t,x,\la)}{a_j(x,\la)}w_k(\tau_j(t,x,x_j,\la),x_j),\\
\left[D^t_0(\la)v\right]_j(t,x)&:=&\sum_{k\not=j}\int_{x_j}^x\d_td_{jk}(t,x,\xi,\la)v_k(\tau_j(t,x,\xi,\la),x_j) d\xi,\\
\left[D^x_0(\la)v\right]_j(t,x)&:=&\sum_{k\not=j}\left[d_{jk}(t,x,x,\la)
u_k(t,x)+\int_{x_j}^x
\d_xd_{jk}(t,x,\xi,\la)v_k(\tau_j(t,x,\xi,\la),\xi) d\xi\right],\\
\left[D^x_1(\la)w\right]_j(t,x)&:=&-\sum_{k\not=j}\int_{x_j}^x
\frac{d_{jk}(t,x,\xi,\la)}{a_j(x,\la)}w_k(\tau_j(t,x,\xi,\la),\xi) d\xi.
\end{eqnarray*}
Therefore,
\begin{eqnarray*}
\d_t(\D(\la)\CC(\la)v)&=&D_0^t(\la)\CC(\la)v+\D(\la)\left(C^t_0(\la)v+\CC(\la)\d_tv\right),\\
\d_t(\D(\la)^2v)&=&D_0^t(\la)\D(\la)v+\D(\la)\left(D^t_0(\la)v+\D(\la)\d_tv\right),\\
\d_x(\D(\la)\CC(\la)v)&=&D_0^x(\la)\CC(\la)v+D_1^x(\la)\left(C^t_0(\la)v+\CC(\la)\d_tv\right),\\
\d_x(\D(\la)^2v)&=&D_0^x(\la)\D(\la)v+D_1^x(\la)\left(D^t_0(\la)v+\D(\la)\d_tv\right).
\end{eqnarray*}
We have to show that for all $\la \in [-\eps_0,\eps_0]$ and $v \in V$ it holds  $\D(\la)\CC(\la)\d_tv \in U_0$, $\D(\la)^2\d_tv \in U_0$,
$D_1^x(\la)\CC(\la)\d_tv \in U_0$, $D_1^x(\la)\D(\la)\d_tv \in U_0$ and
$$
\|\D(\la)\CC(\la)\d_tv\|_0+\|\D(\la)^2\d_tv\|_0+\|D_1^x(\la)\CC(\la)\d_tv\|_0+\|D_1^x(\la)\D(\la)\d_tv\|_0
\le \mbox{const}\|v\|_0.
$$

Let us start with  the term $\D(\la)^2 \d_tv$. We have (cf. \reff{dDdef})
$$
\left[\D(\la)^2 \d_tv\right]_j(t,x)=
\sum_{l\not=k\not=j}\int_{x_j}^x\int_{x_k}^\xi d_{jkl}(t,x,\xi,\eta,\la)\d_tv_l(\tau_k(\tau_j(t,x,\xi,\lambda),\xi,\eta,\la),\eta)\,d\eta\, d\xi
$$
with
$
d_{jkl}(t,x,\xi,\eta\la):=d_{jk}(t,x,\xi,\la)d_{kl}(\tau_j(t,x,\xi,\lambda),\xi,\eta).
$
Moreover, taking into account  \reff{hyp0}, \reff{hyp} and \reff{charom},
the following identity is true:
$$
\frac{d}{d\xi}v_l(\tau_k(\tau_j(t,x,\xi,\lambda),\xi,\eta,\lambda),\eta)
=\left(\frac{1}{a_j(\xi,\la)}-\frac{1}{a_k(\xi,\la)}\right)
\d_tv_l(\tau_k(\tau_j(t,x,\xi,\lambda),\xi,\eta,\lambda),\eta).
$$
It follows that
\begin{eqnarray*}
&&\int_{x_j}^x\int_{x_k}^\xi d_{jkl}(t,x,\xi,\eta,\la)\d_tv_l(\tau_k(\tau_j(t,x,\xi,\lambda),\xi,\eta,\la),\eta)d\eta d\xi\\
&&=
\int_{x_k}^x\int_\eta^{x_j}
\frac{a_j(\xi,\la)a_k(\xi,\la)}{a_k(\xi,\la)-a_j(\xi,\la)}
d_{jkl}(t,x,\xi,\eta,\la)\frac{d}{d\xi}v_l(\tau_k(\tau_j(t,x,\xi,\lambda),\xi,\eta,\lambda),\eta)d\xi d\eta.
\end{eqnarray*}
Integrating by parts in $\xi$, we see that the absolute values of these
integrals   
can be estimated by a constant times  $\|v\|_0$, where the constant does not depend on
$x$, $t$ and $\la$.

Now, let us consider the term $\D(\la)\CC(\la) \d_tv$. We have (cf. \reff{CCdef} and  \reff{dDdef})
\begin{eqnarray*}
&&\left[\D(\la)\CC(\la) \d_tv\right]_j(t,x)\\
&&=\sum_{k=1\atop k\not=j}^m\sum_{l=m+1}^n\int_{x_j}^xd_{jk}(t,x,\xi,\la))
{c}_{kl}(\tau_k(t,x,\xi,\la),\xi,\la)\d_tv_l(\tau_j(\tau_k(t,x,\xi,\la),\xi,0,\la),0)d\xi\\
&&+\sum_{k=m+1\atop k\not=j}^n\sum_{l=1}^m\int_{x_j}^xd_{jk}(t,x,\xi,\la))
{c}_{kl}(\tau_k(t,x,\xi,\la),\la)\d_tv_l(\tau_j(\tau_k(t,x,\xi,\la),1,\xi,\lambda),1)d\xi.
\end{eqnarray*}
As above, we see that the absolute values of these integrals 
can be estimated by a constant times  $\|v\|_0$, where the constant does not depend on
$x$, $t$ and $\la$.

Similarly one shows that  for all $\la\in[-\eps_0,\eps_0]$ and  $v \in V$ we have $D_1^x(\la)\CC(\la)\d_tv \in U_0$, $D_1^x(\la)\D(\la)\d_tv \in U_0$ and
$
\|D_1^x(\la)\CC(\la)\d_tv\|_0+\|D_1^x(\la)\D(\la)\d_tv\|_0
\le \mbox{const}\|v\|_0.
$
\end{proof}

\begin{cor} 
\label{corro1}
For all $\la \in [-\eps_1,\eps_1]$ the following is true:

(i) The operator $I-\CC(\la)-\D(\la)$ is Fredholm of index zero from $U_0$ to $U_0$.

(ii) $\ker (I-\CC(\la)-\D(\la)) \subset V$.
\end{cor}
\begin{proof}
The assertions (i) and (ii) follow from the assertions (i) and (ii) of Lemma~\ref{B2}, respectively.
\end{proof}

\begin{lem}
\label{B5anw}
The operator $I-\CC(0)-\D(0)$ is injective.
\end{lem}
\begin{proof}
Suppose $(I-\CC(0)-\D(0))u=0$ for some $u \in U_0$. Then $u \in V$ by Corollary~\ref{corro1}(ii).
Hence, $u$ is a solution of  the linearized problem
\reff{evp}. But, by the assumption of Theorem \ref{thm:hopf}, this problem does not have nontrivial solutions.
\end{proof}

\begin{cor}
\label{corro}
For all nonnegative integers $l$ it holds $u_0 \in U_l$.
\end{cor}
\begin{proof}
Because of Lemma \ref{prop}  the map $F_0:\R \times U_0 \to U_0$, which is defined by 
$$
F_0(s,u):=T(s)(I-C(0,T(-s)u)T(-s)u-D(0,T(-s)u,T(-s)u)),
$$
is $C^\infty$-smooth. We have  $F_0(s,T(s)u_0)=0$ for all $s \in \R$, in particular,  $F_0(0,u_0)=0$.
Moreover, \reff{dv} yields
$\d_uF_0(0,u_0)=I-\CC(0)-\D(0)$.
Hence, from Corollary \ref{corro1}(i) and
Lemma  \ref{B5anw}  it follows that $\d_uF_0(0,u_0)$ is an isomorphism from $U_0$ to $U_0$, and
the classical implicit function theorem yields that the map $s \approx 0
\mapsto T(s)u_0 \in U_0$ is  $C^\infty$-smooth, i.e., $u_0 \in U_l$ for all $l$.
\end{proof}

The following two remarks point at two technicalities of our approach:
\begin{rem}\rm
\label{tech}
The function $v$ is artificially introduced in  \reff{rep1}--\reff{rep2}: If $(\la,u)$ is a solution to  \reff{rep1}--\reff{rep2} for some $v$, then it is for any $v$.

For our purposes we use two choices of $v$: $v=u$ and $v=u_0$.

First, in Corollary \ref{corro} we took $v=u$. Using the facts  that $u=T(s)u_0$ is a solution to the equation $u-T(s)\left(C(0,T(-s)u)T(-s)u+D(0,T(-s)u,T(-s)u)\right)=0$, 
that the left-hand side of this equation depends smoothly on $s$ and that the linearization of this  left-hand side in $u=u_0$, namely
$$
\frac{d}{du}[u-C(0,u)u-D(0,u,u)]_{u=u_0}=I-\CC(0)-\D(0),
$$
is Fredholm of index zero from $U_0$ to $U_0$, we proved that $u_0 \in U_l$ for all $l$.
But, unfortunately, we do not know if  $\frac{d}{du}[u-C(\la,u)u-D(\la,u,u)]_{u=u_0}$ is Fredholm of index zero from $U_0$ to $U_0$ for all $\la \approx 0$,
therefore the choice $v=u$ is not useful for further purposes.

Later on we will take $v=u_0$. Using the already known property that  $u_0 \in U_l$ for all $l$, now we can prove that  $T(s)\left(C(0,T(-s)u_0)T(-s)u+D(0,T(-s)u,T(-s)u_0)\right)$ 
depends smoothly on $s$. Moreover, due to Corollary \ref{corro1} the operator
$$
\frac{d}{du}[u-C(\la,u_0)u-D(\la,u,u_0)]_{u=u_0}=I-\CC(\la)-\D(\la)
$$
is Fredholm of index zero from $U_0$ to $U_0$.
Hence, we can apply Theorem~\ref{thm:IFT} to the equation $u=C(\la,u_0)u+D(\la,u,u_0)$.
\end{rem}

\begin{rem}\rm
If $u$ is a (classical) solution to \reff{eq:1.1}--\reff{eq:1.2}, 
then $u$ is a solution to  \reff{rep1}--\reff{rep2} for any $v \in U_0$. But we do not know 
if the converse is true: If a 
function $u \in U_0$ solves  \reff{rep1}--\reff{rep2} for some (and hence for any) $v \in U_0$, is then $u$ a $C^1$-smooth function, and hence a solution to  \reff{eq:1.1}--\reff{eq:1.2}?
This is the reason why in Theorem \ref{thm:hopf} we suppose that $u_0$ is a classical solution to  
\reff{eq:1.1}--\reff{eq:1.2}. It seems to be not sufficient to suppose that 
$u_0$ is only a continuous function satisfying
$u_0-C(\la,v)u_0-D(\la,u_0,v)=0$ for  some $v$. 

The $C^1$-smoothness  of $u_0$ is used in the proof of Lemma \ref{B3anw}.
 Because $u_0$ is  $C^1$-smooth, the coefficient 
functions $c_{jk}$,  $d_{jk}$ and  $d_{jkl}$
are   $C^1$-smooth, and therefore one can integrate by parts.
\end{rem}

Now we represent the problem  (\ref{eq:1.1})--(\ref{eq:1.2}) 
in the setting of Section \ref{IFT}. Define
$$
F: [-\eps_0,\eps_0] \times U_0\to U_0:\; F(\la,u):=u-C(\la,u_0)u-D(\la,u,u_0).
$$
Then, on the account of Lemma \ref{prop}  and Corollary \ref{corro}, the map
$$
(s,u) \in \R \times U_0 \mapsto \F(\la,s,u):=T(s)F(\la,T(-s)u) \in U_0
$$
is $C^\infty$-smooth for all $\la \in  [-\eps_0,\eps_0]$, i.e., the  condition \reff{infsmooth} is fulfilled.
Moreover, using the definitions \reff{cdef}, \reff{Cdef} and \reff{Fdef}, it is easy to show that the conditions \reff{ksmooth} and \reff{apriori} are fulfilled.
Further, Lemma \ref{corro1} (i) yields the condition \reff{FH}. And  finally, Lemmas \ref{B4anw}, 
 \ref{B5anw} and   \ref{B2} and  Remark \ref{B3} (with $U:=U_0$) imply the  condition
\reff{coerz}. Hence, Theorem~\ref{thm:IFT} can be applied to the equation $u=C(\la,u_0)u-D(\la,u,u_0)$, 
and this gives the assertions of Theorem 
\ref{thm:hopf}.

\subsection{More general boundary conditions}
\label{otherbc}

Theorem \ref{thm:hopf} can be generalized to more general boundary conditions
of the type
\beq\label{eq:1.2b}
\begin{array}{l}
\displaystyle
u_j(t,0) = \sum\limits_{k=1}^n\left(r_{jk}^{00}(t,\la)u_k(t,0)+r_{jk}^{01}(t,\la)u_k(t,1)\right)
,\quad  j=1,\ldots,m,\\
\displaystyle
u_j(t,1) = \sum\limits_{k=1}^n\left(r_{jk}^{10}(t,\la)u_k(t,0)+r_{jk}^{11}(t,\la)u_k(t,1)\right)
, \quad  j=m+1,\ldots,n.
\end{array}
\ee
Here $r_{jk}^{\alpha \beta}:\R \times [-\eps_0,\eps_0] \to \R$ are $C^\infty$-smooth and $2\pi$-periodic with respect to time.

Reflection boundary conditions of the type  \reff{eq:1.3} appear, for example, in semiconductor laser modeling (see \cite{LiRadRe,Rad,RadWu,Sieber}).
Boundary  conditions of the type  \reff{eq:1.2b} appear, for example, in boundary feedback control problems, see \cite{Pavel} for applications.

If  \reff{eq:1.3} is replaced by  \reff{eq:1.2b}, then \reff{rep1}--\reff{rep2} is replaced by 
\begin{eqnarray*}
\lefteqn{
u_j(t,x)=c_j(t,x,0,\la,v)\sum_{k=1}^n\left(r_{jk}^{00}(\tau_j(t,x,0,\la),\la)u_k(\tau_j(t,x,0,\la),0)\right.}\\
&&\left.+r_{jk}^{01}(\tau_j(t,x,0,\la),\la)u_k(\tau_j(t,x,0,\la),1)\right)\nonumber\\
&&+\int_0^x \frac{c_j(t,x,\xi,\la,v)}{a_j(\xi,\la)}f_j(\tau_j(t,x,\xi,\la),\xi,\lambda,u(\tau_j(t,x,\xi,\lambda),\xi),v(\tau_j(t,x,\xi,\lambda),\xi)d\xi\nonumber\\
&&\mbox{ for } j=1,\ldots,m,
\end{eqnarray*}
\begin{eqnarray*}
\lefteqn{
u_j(t,x)=c_j(t,x,1,\la,v)\sum_{k=1}^m\left(r_{jk}^{10}(\tau_j(t,x,1,\la),\la)u_k(\tau_j(t,x,1,\la),0)\right.}\\
&&\left.+r_{jk}^{11}(\tau_j(t,x,1,\la),\la)u_k(\tau_j(t,x,1,\la),1)\right)\\
&&-\int_x^1 \frac{c_j(t,x,\xi,\la,v)}{a_j(\xi,\la)}f_j(\tau_j(t,x,\xi,\la),\xi,\la,u(\tau_j(t,x,\xi,\la),\xi),v(\tau_j(t,x,\xi,\la),\xi)d\xi\nonumber\\
&&\mbox{ for } j=m+1,\ldots,n.
\end{eqnarray*}
and \reff{CCdef} 
is replaced by 
\begin{eqnarray*}
\left(\CC(\la)u\right)_j(t,x)&=&
c_{j}(t,x,0,\la,u_0)\sum_{k=1}^n\left(r_{jk}^{00}(\tau_j(t,x,0,\la),\la)u_k(\tau_j(t,x,0,\la),0)\right.\\
&&\left.+r_{jk}^{01}(\tau_j(t,x,0,\la),\la)u_k(\tau_j(t,x,0,\la),1)\right)
\mbox{ for }j=1,\ldots,m,
\end{eqnarray*}
\begin{eqnarray*}
\left(\CC(\la)u\right)_j(t,x)&=&
c_{j}(t,x,1,\la,u_0)\sum_{k=1}^n\left(r_{jk}^{10}(\tau_j(t,x,1,\la),\la)u_k(\tau_j(t,x,1,\la),0)\right.\\
&&\left.+r_{jk}^{11}(\tau_j(t,x,1,\la),\la)u_k(\tau_j(t,x,1,\la),1)\right)
\mbox{ for }
j=m+1,\ldots,n.
\end{eqnarray*}
In order to have a replacement for Lemma \ref{prop1}, we need that  there exist $\eps_1 \in (0,\eps_0]$ and $d>0$ such that for all $\la \in  [-\eps_1,\eps_1]$ the operator
$I-\CC(\la)$ is bijective from $U_0$ to $U_0$ and $\left\|(I-\CC(\la))^{-1}\right\|_{\LL(U_0)} \le d$. A sufficient condition (which is far from being necessary) for that 
is 
\beq
\label{suff}
\|\CC(\la)\|_{\LL(U_0)} \le \mbox{const} <1 \mbox{ for all } \la \approx 0.
\ee
Remark that the condition $\|\CC(0)\|_{\LL(U_0)} <1$ for all  $\la \approx 0$ is not sufficient for \reff{suff}, in general, because the map 
$\la \mapsto \CC(\la)$ is not continuous with respect to the
uniform operator norm in $\LL(U_0)$, in general. But another ``$\la$-independent'' condition is  sufficient for \reff{suff}, namely
\begin{eqnarray*}
\max_{t,x}c_{j}(t,x,0,0,u_0) \sum_{k=1}^n\left(\max_{t}|r_{jk}^{00}(t,0)|+\max_{t}|r_{jk}^{01}(t,0)|\right)<1 && \mbox{for }j=1,\ldots,m,\\
\max_{t,x}c_{j}(t,x,1,0,u_0) \sum_{k=1}^n\left(\max_{t}|r_{jk}^{10}(t,0)|+\max_{t}|r_{jk}^{11}(t,0)|\right)<1 && \mbox{for }j=m+1,\ldots,n.
\end{eqnarray*}
But the definition \reff{cdef} yields that $c_j(t,x,0,0,u_0) \le 1$ for $j=1,\ldots,m$ and  $c_j(t,x,1,0,u_0) \le 1$ for $j=m+1,\ldots,n$ if
\beq\label{cabsch}
\begin{array}{l}
\displaystyle
\min_{t,x}\frac{\d_{u_j}b_j(t,x,u_0(t,x),0)}{a_j(x,0)}>0 \mbox{ for } j=1,\ldots,m,\\
\displaystyle\max_{t,x}\frac{\d_{u_j}b_j(t,x,u_0(t,x),0)}{a_j(x,0)}<0 \mbox{ for } j=m+1,\ldots,n.
\end{array}
\ee
Therefore, if \reff{cabsch} is true and if 
\begin{eqnarray*}
\sum_{k=1}^n\left(\max_{t}|r_{jk}^{00}(t,0)|+\max_{t}|r_{jk}^{01}(t,0)|\right)<1 && \mbox{for }j=1,\ldots,m,\\
\sum_{k=1}^n\left(\max_{t}|r_{jk}^{10}(t,0)|+\max_{t}|r_{jk}^{11}(t,0)|\right)<1 && \mbox{for }j=m+1,\ldots,n,
\end{eqnarray*}
then \reff{suff} is true.

\section{Time-periodic solutions to second-order hyperbolic equations}
\label{wave}
\renewcommand{\theequation}{{\thesection}.\arabic{equation}}
\setcounter{equation}{0}

\subsection{Setting and main result}
\label{setting3}

In this section we consider the boundary value problem (\ref{eq:1.1a})--(\ref{eq:1.2a}),
where the coefficient
functions $a: [0,1]\times [-\eps_0,\eps_0]\to\R$ and 
$b: \R \times [0,1] \times [-\eps_0,\eps_0]\times \R^3\to\R$
are  $C^\infty$-smooth, 
$b$ is $2\pi$-periodic in $t$ and 
\beq \label{hyp0a}
a(x,\la)> 0  \mbox{ for all }   t \in \R,\,  x\in[0,1]   \mbox{ and } \la \in  [-\eps_0,\eps_0].
\ee 
Speaking about solutions to  (\ref{eq:1.1a})--(\ref{eq:1.2a}), 
we will mean classical solutions again.

Let $u_0$ be a 
solution to  (\ref{eq:1.1a})--(\ref{eq:1.2a}) with $\la=0$.
Denote
\beq
\label{bjdef}
\begin{array}{rcl} 
b_0(t,x)&:=&\d_{4}b(t,x,0,u_0(x,t),\d_tu_0(t,x),\d_xu_0(t,x)),\\
b_1(t,x)&:=&\d_{5}b(t,x,0,u_0(t,x),\d_tu_0(t,x),\d_xu_0(t,x)),\\
b_2(t,x)&:=&\d_{6}b(t,x,0,u_0(t,x),\d_tu_0(t,x),\d_xu_0(t,x))
\end{array} 
\ee
and
\begin{eqnarray*}
\lefteqn{
R_0(t):=\int_0^1\frac{1}{a(\eta,0)}\left(b_1\left(t-\int_0^\eta\frac{dx}{a(x,0)},\eta\right)+b_1\left(t+\int_0^\eta\frac{dx}{a(x,0)},\eta\right)\right)d\eta}\\
&&+\int_0^1\frac{1}{a(\eta,0)^2}\left(b_2\left(t-\int_0^\eta\frac{dx}{a(x,0)},\eta\right)-b_2\left(t+\int_0^\eta\frac{dx}{a(x,0)},\eta\right)\right)d\eta,
\end{eqnarray*}
\begin{eqnarray*}
\lefteqn{
S_0(t):=\int_0^1\frac{1}{a(\eta,0)}\left(b_1\left(t+\int_\eta^1\frac{dx}{a(x,0)},\eta\right)+b_1\left(t-\int_\eta^1\frac{dx}{a(x,0)},\eta\right)\right)d\eta}\\
&&+\int_0^1\frac{1}{a(\eta,0)^2}\left(b_2\left(t+\int_\eta^1\frac{dx}{a(x,0)},\eta\right)-b_2\left(t-\int_\eta^1\frac{dx}{a(x,0)},\eta\right)\right)d\eta.
\end{eqnarray*}
Here and in what follows we denote by $\d_jb$ ($j=4,5,6$) the partial derivative of $b$ with respect to the $j$-th argument, i.e.,
with respect to $u$ for $j=4$, with respect to $\d_tu$ for $j=5$ and with respect to $\d_xu$ for $j=6$.
 
\begin{thm}
\label{thm:hopfa}
Suppose \reff{hyp0a} and 
\beq
\label{Ra}
R_0(t)\not=0 \mbox{ for all } t \mbox{ or } S_0(t)\not=0 \mbox{ for all } t.
\ee
Assume that the linearized problem
\beq
\label{evpa}
\begin{array}{r}
\partial_t^2u  - a(x,0)^2\partial^2_xu  + b_0(t,x)u+ b_1(t,x)\d_tu+ b_2(t,x)\d_xu=0,\\  
u(t+2\pi,x)=u(t,x), \\
u(t,0)=\d_xu(t,1)=0
\end{array}
\ee
does not have a solution $u \not= 0$.

Then there exist $\eps>0$, $\delta >0$ and a $C^\infty$-map
$\hat u : \R \times [0,1] \times[-\eps,\eps] \to \R$  such that the following is true:

(i) For all $\la\in [-\eps,\eps]$ the function $\hat u(\cdot,\cdot,\la)$
is a solution  to (\ref{eq:1.1a})--(\ref{eq:1.2a}).

(ii) If $u$ is a solution to (\ref{eq:1.1a})--(\ref{eq:1.2a}) with 
$$ 
|\la|+ \max_{t,x}|u(x,t)-u_0(t,x)|+\max_{t,x}|\d_tu(x,t)-\d_tu_0(t,x)|+\max_{t,x}|\d_xu(x,t)-\d_xu_0(t,x)|<\de,
$$
then
$u(t,x)=\hat u(t,x,\la)$ for all  $t\in\R$ and $x\in[0,1]$.
In particular, $u_0=\hat u(\cdot,\cdot,0)$ is  $C^\infty$-smooth.
\end{thm}

\begin{rem}\rm
\label{telegraph}
Special cases of equations of the type \reff{eq:1.1a} are nonlinear telegraph equations of the type
$$
\d_t^2u-a(x,\la)^2\d_x^2u+a_1(x,\la)\d_tu+a_2(x,\la)\d_xu+b(t,x,\la,u)=0
$$
with $a(x,\la)>0$ and $a_1(x,\la)>0$. Then
$$
R_0(t)=S_0(t)=2\int_0^1\frac{a_1(x,0)}{a(x,0)}dx\not=0,
$$
i.e. assumption \reff{Ra} is fulfilled independently of the choice of the functions $a_2$ and $b$ and of the starting solution $u_0$.
\end{rem}

\begin{rem}\rm
\label{lit}
Let us mention some results  related to Theorem \ref{thm:hopfa}.

In \cite{Rab2} P. H. Rabinowitz investigated fully nonlinear dissipative wave equations of the type
$$
\d_t^2u-\d_x^2u+a_0\d_tu+ \la b(t,x,u,\d_tu,\d_xu,\d_t^2,\d_t\d_xu,\d_x^2)=0, \; x \in (0,1)
$$
with homogeneous Dirichlet boundary conditions.
By means of Nash-Moser iterations
 it was shown that  
under natural assumptions for all $\la \approx 0$ there exists  
a time-periodic solution close to zero.
Using Schauder's fixed point theorem,  
M. \v{S}t\v{e}dr\'{y} \cite{Stedry}   generalized this result to the case of bounded domains of any dimension 
and to differential operators of higher order in $x$.
Both authors did not consider the question of smoothness of the data-to-solution map
and the question of time-periodic solutions far from zero.

In \cite{Craig} W. Craig considered  fully nonlinear dissipative wave equations of the type
$$
\d_t^2u-\d_x^2u+a_0\d_tu-(\la_0+\la)u+ b(t,x,u,\d_tu,\d_xu,\d_t^2u,\d_t\d_xu,\d_x^2u)=0, \; x \in (0,1)
$$
with  homogeneous Dirichlet boundary conditions (as well as their analogues with higher space dimension),
where $b(t,x,\cdot)$ is at least of quadratic order. It was shown (by means of a 
Liapunov-Schmidt reduction and Nash-Moser iterations) that  under natural assumptions 
there exists a branch of non-trivial   time-periodic solutions bifurcating at $\la=0$ from the 
solution $u=0$, which can be Lipschitz continuously parameterized   
by $\la \approx 0$.

In \cite{Raugel, Raugel1} J. Hale and G. Raugel considered systems of damped autonomous wave equations of the type
$$
\d_t^2u_j-\Delta u_j+(a_0+\la a_1(x))\d_tu_j+ a_2u_j+b_j(u)=0, \; x \in \Omega,\; j=1,2
$$
with homogeneous Dirichlet or Neumann boundary conditions in an $n$-dimensional bounded domain $\Omega$. By means of the Leray-Schauder fixed point theorem it was 
shown 
that under natural assumptions a non-degenerate periodic orbit persists
uniquely under small perturbations of $\la$. The question of smoothness of the data-to-solution map was not addressed.

In \cite{Kolesov} A.~Yu.~Kolesov and N.~Ch.~Rozov proved  existence of small time-periodic solutions to  parametrically excited damped wave equations of the type
$$
\d_t^2u+\eps\d_tu+(1+\eps\alpha \cos 2 \omega t)u=\eps^2a^2\d_x^2u+f(u,\d_tu),\; x \in (0,\pi)
$$
with homogeneous Dirichlet boundary conditions, where $0<\eps<<1$ is a small singular perturbation parameter, and $f$ and its first partial derivatives vanish at zero.

Of course, if assumption \reff{Ra} is dropped, then one should expect a
solution behavior which is completely different to Theorem \ref{thm:hopfa}, see, for example, the results of  M. Berti and M. Procesi in \cite{Procesi}
on periodically forced completely resonant nonlinear wave equations. 
\end{rem}

\begin{rem}\rm
\label{auto}
It is known how to prove a so-called $G$-invariant implicit function theorems for equivariant equations by means of classical  implicit function theorems
(see, e.g. \cite{Dancer}) and how to apply this to periodic solutions of autonomous evolution equations.
Hence, Theorem~\ref{thm:IFT} can be translated into an equivariant setting (i.e. with $T(s)F(\la,u)=F(\la,T(s)u)$ for all $s,\la,u$), and Theorems~\ref{thm:hopf} 
and \ref{thm:hopfa} 
can be translated into an  autonomous setting 
(in this case the local uniqueness assertion should be understood modulo time shifts, of course),
this will be done in \cite{KRaut}.

There is an essential  difference between the autonomous and the  nonautonomous cases.
 In the  nonautonomous case  the proof of smoothness of the data-to-solution map is 
essentially simpler if one assumes that the coefficient  $a(x,\la)$ is  $\la$-independent (cf. Remark~\ref{nonsm}).
This simplification does not occur in the autonomous case, i.e., for equations of the type
$$
\d_t^2u-a(x,\la)^2\d_x^2u+b(x,\la,u,\d_tu,\d_xu)=0.
$$
The reason is that one should scale the time in order to work in spaces of periodic functions with fixed period. 
Then, after time scaling, the unknown frequency $\om$ appears explicitely in the
main part of the equation, namely
$$
\om^2\d_t^2u-a(x,\la)^2\d_x^2u+b(x,\la,u,\om\d_tu,\d_xu)=0,
$$
and this way a coefficient dependence appears in the main part of the differential operator even if $a(x,\la)$ is $\la$-independent.
\end{rem}

\subsection{Proof of Theorem \ref{thm:hopfa}}
\label{Proof4.1}

In this subsection we will prove Theorem \ref{thm:hopfa}, hence we will suppose that all its assumptions are fulfilled.

Denote by $U_0$ the space of all continuous maps $v:[0,1] \times \R \to \R^2$ such that $v(\cdot,x)$ is $2\pi$-periodic,
with the norm 
$$
\|v\|_0:=\max_{j=1,2} \max_{t,x} |v_j(t,x)|.
$$
Given $v \in U_0$, write
\begin{eqnarray*}
(J_0^\la v)(t,x)&:=&\frac{1}{2}\int_0^x\frac{v_1(t,\xi)-v_2(t,\xi)}{a(\xi,\la)}d\xi,\\
(J_1v)(t,x)&:=&\frac{v_1(t,x)+v_2(t,x)}{2},\\
(J_2^\la v)(t,x)&:=&\frac{v_1(t,x)-v_2(t,x)}{2a(x,\la)}.\\
\end{eqnarray*}

\begin{lem}
\label{soleq}
(i) Let $u$ be a solution to  (\ref{eq:1.1a})--(\ref{eq:1.2a}). Then the pair
\beq
\label{vw}
v_1(t,x):=\d_tu(t,x)+a(x,\la)\d_xu(t,x),\; v_2(t,x):=\d_tu(t,x)-a(x,\la)\d_xu(t,x)
\ee
is a solution to
\begin{eqnarray}
\label{Gl}
&&\d_tv_1-a(x,\la)\d_xv_1=\d_tv_2+a(x,\la)\d_xv_2\nonumber\\
&&=-\frac{\d_xa(x,\la)}{2}(v_1-v_2)-b\left(t,x,\la,J_0^\la v,J_1 v,J_2^\la v\right),\\
\label{per}
&&v_1(t+2\pi,x)-v_1(t,x)=v_2(t+2\pi,x)-v_2(t,x)=0,\\
\label{bcond}
&&v_1(t,0)+v_2(t,0)=v_1(t,1)-v_2(t,1)=0.
\end{eqnarray} 

(ii) Let $v=(v_1,v_2)$ be a solution to \reff{Gl}-\reff{bcond}. Then $u=J^\la_0 v$
is a solution to  (\ref{eq:1.1a})--(\ref{eq:1.2a}).
\end{lem}
\begin{proof}
(i) From \reff{vw} it follows $\d_tv_1=\d_t^2u+a\d_t\d_xu$ and $\d_xv_1=\d_t\d_xu+\d_xa\d_xu+a\d^2_xu$. Hence,  (\ref{eq:1.1a}) yields
\beq
\label{u_1}
\d_tv_1-a(x,\la)\d_xv_1=-a(x,\la)\d_xa(x,\la)\d_xu-b\left(t,x,\la,u,\d_tu,\d_xu\right).
\ee
But  \reff{vw} yields also $2a\d_xu=v_1-v_2$. Hence, \reff{eq:1.2a} implies $u=J_0^\la v$.
Inserting this into \reff{u_1}, we get the first  equation from \reff{Gl}. The second one can be proved analogously.
Finally, the boundary conditions \reff{bcond} follow directly from the boundary conditions  \reff{eq:1.2a}.

(ii) From \reff{Gl} it follows $\d_t(v_1-v_2)=a\d_x(v_1+v_2)$. Hence, \reff{bcond} and $u=J_0^\la v$ yield
$2\d_tu(t,x)=\int_0^x(\d_xv_1(t,\xi)+\d_xv_2(t,\xi))d\xi=v_1(t,x)+v_2(t,x),$
i.e.
\beq
\label{d^2}
2\d_t^2u(t,x)=\d_tv_1(t,x)+\d_tv_2(t,x).
\ee
Moreover, from $u=J_0^\la v$ follows $2a\d_xu=v_1-v_2$, i.e.
\beq
\label{d_x}
2a(x,\la)\d_x^2u=\d_xv_1-\d_xv_2-2\d_xa(x,\la)\d_xu.
\ee
But \reff{Gl}, \reff{d^2} and  \reff{d_x} imply \reff{eq:1.1a}. Finally, \reff{bcond} and $u=J_0^\la v$ yield  \reff{eq:1.2a}.
\end{proof}

Now we proceed as in Section \ref{Proof3.1}.
The characteristics of the hyperbolic system \reff{Gl} are given by (cf. \reff{charom})
\beq
\label{charoma}
\tau_j(t,x,\xi,\la):=t+(-1)^j\int_x^\xi\frac{d\eta}{a(\eta,\la)}.
\ee
Let us introduce an artificial linear diagonal term (which depends on a function $w \in U_0$)
and rewrite  the system \reff{Gl} as (cf. \reff{fdef})
\begin{eqnarray*}
&&\d_tv_1-a(x,\la)\d_xv_1+\left(\frac{\d_xa(x,\la)}{2}+\tilde{b}_1(t,x,\la,w)\right)v_1=f_1(t,x,\la,v,w),\\
&&\d_tv_2+a(x,\la)\d_xv_2+\left(-\frac{\d_xa(x,\la)}{2}+\tilde{b}_1(t,x,\la,w)\right)v_2=f_2(t,x,\la,v,w)
\end{eqnarray*}
with
\begin{eqnarray}
\label{tildeb}
\tilde{b}_j(t,x,\la,w)&:=&\frac{1}{2}\d_{5}b\left(t,x,\la,J_0^\la w(t,x),J_1 w(t,x),J_2^\la w(t,x)\right)\nonumber\\
&&-\frac{(-1)^{j}}{2a(x,\la)}\d_{6}b\left(t,x,\la,J_0^\la w(t,x),J_1 w(t,x),J_2^\la w(t,x)\right),\\
\label{f1}
f_1(t,x,\la,v,w)&:=&\frac{\d_xa(x,\la)}{2}v_2(t,x)+\tilde{b}_1(t,x,\la,w)v_1(t,x)\nonumber\\
&&-b\left(t,x,\la,J_0^\la v(t,x),J_1 v(t,x),J_2^\la v(t,x)\right),\\
\label{f2}
f_2(t,x,\la,v,w)&:=&-\frac{\d_xa(x,\la)}{2}v_1(t,x)+\tilde{b}_2(t,x,\la,w)v_2(t,x)\nonumber\\
&&-b\left(t,x,\la,J_0^\la v(t,x),J_1 v(t,x),J_2^\la v(t,x)\right).
\end{eqnarray}
Remark that now the partial derivatives $\d_{v_j}f_j(t,x,\la,v,w)|_{v=w}$ do not vanish like in 
Section~\ref{Proof3.1},
but they are good enough for our purposes, since they are integral operators:
$$
\d_{v_j}f_j(t,x,\la,v,w)|_{v=w}\bar{v}_j=-\frac{(-1)^j}{2}\d_4b\left(t,x,\la,J_0^\la w(t,x),J_1 w(t,x),J_2^\la w(t,x)\right)\int_0^x\bar{v}_j(\xi) d\xi.
$$
Moreover, the coefficients $c_j$ (cf. \reff{cdef}) read
\begin{eqnarray}
\label{c1}
c_1(t,x,\xi,\la,w)&:=&\exp\int_x^\xi\left(-\frac{\d_xa(\eta,\la)}{2a(\eta,\la)}
-\frac{\tilde{b}_1(\tau_1(t,x,\eta,\la),\eta,\la,w)}{a(\eta,\la)}\right)d\eta\nonumber\\
&=&\sqrt{\frac{a(x,\la)}{a(\xi,\la)}}
\exp\int_\xi^x\frac{\tilde{b}_1(\tau_1(t,x,\eta,\la),\eta,\la,w)}{a(\eta,\la)}d\eta,\\
\label{c2}
c_2(t,x,\xi,\la,w)&:=&\exp\int_x^\xi\left(-\frac{\d_xa(\eta,\la)}{2a(\eta,\la)}
+\frac{\tilde{b}_2(\tau_2(t,x,\eta,\la),\eta,\la,w)}{a(\eta,\la)}\right)d\eta\nonumber\\
&=&\sqrt{\frac{a(x,\la)}{a(\xi,\la)}}
\exp\int_x^\xi\frac{\tilde{b}_2(\tau_2(t,x,\eta,\la),\eta,\la,w)}{a(\eta,\la)}d\eta.
\end{eqnarray}

\begin{lem}
\label{soleq1}
The pair  $v=(v_1,v_2)$ satisfies \reff{Gl}-\reff{bcond} if and only if
\begin{eqnarray}
\label{IGl1}
v_1(t,x)&=&c_1(t,x,0,\la,w)v_2(\tau_1(t,x,0,\la),0)\nonumber\\
&&+\int_0^x\frac{c_1(t,x,\xi,\la,w)}{a(\xi,\la)}f_1(t,x,\xi,\la,v,w)d\xi,\\
v_2(t,x)&=&c_2(t,x,1,\la,w)v_1(\tau_2(t,x,1,\la,w),1)\nonumber\\
&&+\int_x^1\frac{c_2(t,x,\xi,\la,w)}{a(\xi,\la)}f_2(t,x,\xi,\la,v,w)d\xi
\label{IGl2}
\end{eqnarray}
for any $w\in U_1$.
\end{lem}
\begin{proof}
The proof can be done by means of integration along characteristics 
like in   Section \ref{Proof3.1} for getting  equivalence of the PDE problem \reff{eq:1.1}-\reff{eq:1.2}
and the system of integral equations  \reff{rep1}-\reff{rep2}.
\end{proof}

Now, for $\la \in[-\eps_0,\eps_0]$ and $w \in U_0$ we define linear bounded operators $C(\la,w): U_0 \to U_0$ (cf. \reff{Cdef}) by
\beq
\label{C}
\left(C(\la,w)v\right)_j(t,x):=
\left\{
\begin{array}{l}
\displaystyle
c_1(t,x,0,\la,w)v_2(\tau_1(t,x,0,\la),0) \mbox{ for }j=1,\\
\displaystyle
-c_2(t,x,1,\la,w)v_1(\tau_2(t,x,1,\la),1) \mbox{ for }
j=2
\end{array}
\right.
\ee
and $C^\infty$-smooth nonlinear operators  $D(\la,\cdot,\cdot): U_0\times U_0 \to U_0$  (cf. \reff{Fdef}) by
\beq
\label{D}
D(\la,v,w)_j(t,x):=(-1)^j\int^x_{x_j}\frac{c_j(t,x,\xi,\la,w)}{a(\xi,\la)}f_j(\tau_j(t,x,\xi,\la),x,\xi,\la,v,w)d\xi
\ee
with $x_1:=0$ and $x_2:=1$.
Using this notation, system \reff{IGl1}--\reff{IGl2} can be written as the operator equation
\beq
\label{abstracta}
v=C(\la,w)v+D(\la,v,w).
\ee

Let us introduce a  $C_0$-group $T(s)$ on $U_0$ by
$
\left(T(s)v\right)(t,x):=v(t+s,x).
$
The corresponding infinitesimal generator is
$
A:=\d_t,
$
and the domains of definition of the powers $A^l$ are
$$
U_l:=\{v \in U_0:\; \d_j^tv \in U_0 \mbox{ for all } j=0,1,\ldots,l\}.
$$
They are Banach spaces with the norms
$$
\|v\|_l:=\sum_{j=0}^l\|\d_t^jv\|_0.
$$

\begin{lem}
\label{propa}
The maps $(s,v,w) \in \R \times U_0^2 \mapsto T(s)C(\la,T(-s)w)T(-s)v \in U_0$ and
$(s,v,w) \in \R \times U_0^2 \mapsto T(s)D(\la,T(-s)v,T(-s)w) \in U_0$ are $C^\infty$-smooth for all $\la \in  [-\eps_0,\eps_0]$.
\end{lem}
\begin{proof}
Because of $(J_0^\la T(s)v)(t,x)=(J_0^\la v)(t+s,x)$, $(J_1 T(s)v)(t,x)=(J_1 v)(t+s,x)$  and  $(J_2^\la T(s)v)(t,x)=(J_2^\la v)(t+s,x)$, the proof follows the same line as the proof  of Lemma~\ref{prop}.
\end{proof}

Define the function  $v_0 \in U_0$ by
$$
v_0(t,x):=(\d_tu_0(t,x)+a(x,0)\d_xu_0(t,x),\d_tu_0(t,x)-a(x,0)\d_xu_0(t,x))
$$
and, for $\la \in   [-\eps_0,\eps_0]$, the operators $\CC(\la), \D(\la) \in \LL(U_0)$ by 
\beq
\label{CC}
\CC(\la):=C(\la,v_0),\; \D(\la):=\d_vD(\la,v,v_0)|_{v=v_0}.
\ee
Because of \reff{C} the operator $\CC(\la)$ is of ``argument shift type'' with vanishing diagonal part:
$$
\left(\CC(\la)v\right)_j(t,x)=
\left\{
\begin{array}{l}
\displaystyle
-c_1(t,x,0,\la,v_0)v_2(\tau_1(t,x,0,\la),0) \mbox{ for }j=1,\\
\displaystyle
c_2(t,x,1,\la,v_0)v_1(\tau_2(t,x,1,\la),1) \mbox{ for }
j=2.
\end{array}
\right.
$$

\begin{lem}
\label{prop1a}
There exist $\eps_1 \in (0,\eps_0]$ and $d>0$ such that for all $\la \in  [-\eps_1,\eps_1]$ the operator
$I-\CC(\la)$ is bijective from $U_0$ to $U_0$ and $\left\|(I-\CC(\la))^{-1}\right\|_{\LL(U_0)} \le d$.
\end{lem}
\begin{proof}
This is Lemma \ref{prop1} with $m=1, n=2$ and with $r_{12}(t,\la)=-1, r_{21}(t,\la)=1$, and hence 
with ${c}_{12}(t,x,\la)$  replaced by 
$$
-c_1(t,x,0,\la,v_0)=-\sqrt{\frac{a(x,\la)}{a(0,\la)}}
\exp\int_0^x\frac{\tilde{b}_1(\tau_1(t,x,\eta,\la),\eta,\la,v_0)}{a(\eta,\la)}d\eta,
$$
and  with ${c}_{21}(t,x,\la)$  replaced by
$$
c_2(t,x,1,\la,v_0)=\sqrt{\frac{a(x,\la)}{a(1,\la)}}
\exp\int_x^1\frac{\tilde{b}_2(\tau_2(t,x,\eta,\la),\eta,\la,v_0)}{a(\eta,\la)}d\eta,
$$
cf. \reff{c1} and \reff{c2}.
This means that the condition $|{c}_{12}(t,1,0){c}_{21}(\tau_1(t,1,0,0),0,0)|\not=1$ (see \reff{firstcond})
now reads
$$
\int_0^1\frac{\tilde{b}_1(\tau_1(t,1,\eta,0),\eta,0,v_0)
+\tilde{b}_2(\tau_2(\tau_1(t,1,0,0),0,\eta,0),\eta,0,v_0)}{a(\eta,0)}d\eta\not=0 \mbox{ for all } 
t.
$$
On the account of \reff{bjdef} and  \reff{tildeb}, this is equivalent to
\begin{eqnarray*}
\lefteqn{
\int_0^1\left(b_1(\tau_1(t,1,\eta,0),\eta)+b_1(\tau_2(\tau_1(t,1,0,0),0,\eta,0),\eta)\right)d\eta}\\
&&+\int_0^1\frac{b_2(\tau_1(t,1,\eta,0),\eta)-b_2(\tau_2(\tau_1(t,1,0,0),0,\eta,0),\eta)}{a(\eta,0)}d\eta\not=0  \mbox{ for all } t,
\end{eqnarray*}
and, because of \reff{charoma}, this is equivalent to $R_0(t) \not=0$ for all $t$.

Similarly one shows that the condition $|{c}_{21}(t,0,0){c}_{12}(\tau_2(t,0,1,0),1,0)|\not=1$  from 
Lemma~\ref{prop1} with $m=1, n=2$
should be replaced by  $S_0(t) \not=0$.
\end{proof}

Because of \reff{tildeb}--\reff{f2} and \reff{D} the  operator $\D(\la)$ is the sum of two integral operators
$$
\D(\la)=\G(\la)+\HH(\la),
$$
where
$$
\left(\G(\la)v\right)_j(t,x):=
\left\{
\begin{array}{l}
\displaystyle
\int^x_0g_1(t,x,\xi,\la)v_2(\tau_1(t,x,\xi,\la),\xi)d\xi \mbox{ for }j=1,\\
\displaystyle
\int_x^1g_2(t,x,\xi,\la)v_1(\tau_2(t,x,\xi,\la),\xi)d\xi \mbox{ for }
j=2
\end{array}
\right.
$$
with
\begin{eqnarray*}
g_j(t,x,\xi,\la)&:=&(-1)^{j}\frac{c_j(t,x,\xi,\la,v_0)}{2a(\xi,\la)}\left[\d_5b(\tau,x,\xi,\la,J_0^\la v_0(\tau,x),J_1 v_0(\tau,x),
J_2^\la v_0(\tau,x))\right.\\
&&\left.-\frac{\d_6b(\tau,x,\xi,\la,J_0^\la v_0(\tau,x),J_1 v_0^\la(\tau,x),J_2^\la v_0(\tau,x))}{a(\xi,\la)}\right]_{\tau=\tau_j(t,x,\xi,\la)}
\end{eqnarray*}
and
$$
\left(\HH(\la)v\right)_j(t,x):=\int_x^{x_j}\int_0^\xi h_j(t,x,\xi,\eta,\la)\left(v_1(\tau_j(t,x,\eta,\la),\eta)+v_2(\tau_j(t,x,\eta,\la),\eta)\right)d\eta d\xi
$$
with
\begin{eqnarray*}
&&h_j(t,x,\xi,\eta,\la)\\
&&:=\frac{(-1)^jc_j(t,x,\xi,\la,v_0)}{2a(\xi,\la)}\Big[\d_4b(\tau,x,\xi,\la,J_0^\la v_0(\tau,x),J_1 v_0(\tau,x),
J_2^\la v_0(\tau,x))\Big]_{\tau=\tau_j(t,x,\xi,\la)}.
\end{eqnarray*}
The operator $\G(\la)$ is a so-called partial integral operator (with  one integration only, but working on functions with two arguments, cf. \cite{Appell1})
with vanishing diagonal part, while the operator $\HH(\la)$ is an integral operator.

Finally, let us introduce the Banach space $V$ with the norm $\|\cdot\|_V$ and the Hilbert space $H$ 
with the scalar product $\langle \cdot,\cdot \rangle$ as follows:
\begin{eqnarray*}
&&V:=\{v \in U_0: \; \d_tv,\d_xv \in U_0\},\; \|v\|_V:= \|v\|_0+ \|\d_tv\|_0+ \|\d_xv\|_0,\\
&&H:=L^2\left((0,1) \times (0,2\pi);\R^2\right), \; \langle v,w \rangle := \int_0^{2\pi}\int_0^1v(t,x)w(t,x)dxdt.
\end{eqnarray*}

\begin{lem}
\label{B3anwa}
It holds
$$
\lim_{\la \to 0} \sup_{\|v\|_0\le 1}\langle \left(\CC(\la)+\D(\la)-\CC(0)-\D(0)\right)v,w\rangle=0 \mbox{ for all } w \in H.
$$
\end{lem}
\begin{proof}
We follow the proof of Lemma \ref{B3anw}. In particular, the scalar product $\langle \HH(\la)v,w\rangle$ equals 
$$
\sum_{j=1}^2\int_0^1\int_x^{x_j}\int_0^\xi\int_0^{2\pi} h_j(t,x,\xi,\eta,\la)\Big[v_1(\tau,\eta)+v_2(\tau,\eta)\Big]_{\tau=\tau_j(t,x,\eta,\la)}w_j(t,x) dt d\eta d\xi dx.
$$
Hence, the change of the integration variables $t\mapsto s$ by means of  
$t=t_j(\tau_j(s,x,\eta,0),x,\eta,\la)$  yields
$$
\lim_{\la \to 0} \sup_{\|v\|_0\le 1}\langle \left(\HH(\la)-\HH(0)\right)v,w\rangle=0 \mbox{ for all } w \in H.
$$
\end{proof}

\begin{lem}
\label{B4anwa}
There exists $d>0$ such that for all $\la \in [-\eps_0,\eps_0]$ and $v \in V$ it holds
$\D(\la)^2v \in V$, $\D(\la)\CC(\la)v \in V$ and  
$$
\|\D(\la)^2v\|_V + \|\D(\la)\CC(\la)v\|_V \le d\|v\|_0.
$$
\end{lem}
\begin{proof}
As in the proof of Lemma  \ref{B4anw}, one can show that there exists $d>0$ such that for all $\la \in [-\eps_0,\eps_0]$ and $v \in V$ it holds
$\G(\la)^2v \in V$, $\G(\la)\CC(\la)v \in V$ and  
$$
\|\G(\la)^2v\|_V + \|\G(\la)\CC(\la)v\|_V \le d\|v\|_0.
$$
Hence, it remains to show that  there exists $d>0$ such that for all $\la \in [-\eps_0,\eps_0]$ and $v \in V$ it holds
$\HH(\la)v \in V$ and  
\beq
\label{Hest1}
\|\HH(\la)v\|_V \le d\|v\|_0.
\ee
To this end, consider $\left(\HH(\la)v\right)_1(t,x)$
(and similarly for  $\left(\HH(\la)v\right)_2(t,x)$):
\begin{eqnarray}
&&\left(\HH(\la)v\right)_1(t,x)\nonumber\\
&&=-\int^x_0\int_0^\xi h_j(t,x,\xi,\eta,\la)\left(v_1(\tau_1(t,x,\xi,\la),\eta)+v_2(\tau_1(t,x,\xi,\la),\eta)\right)d\eta d\xi\nonumber\\
&&=-\int_0^x\int_\eta^x h_j(t,x,\xi,\eta,\la)\left(v_1(\tau_1(t,x,\xi,\la),\eta)+v_2(\tau_1(t,x,\xi,\la),\eta)\right)d\xi d\eta\nonumber\\
&&=\int_0^x\int_\eta^x a(\om(t,x,\zeta,\la),\la) h_j(t,x,\om(t,x,\zeta,\la),\eta,\la)\left(v_1(\zeta,\eta)+v_2(\zeta,\eta)\right)d\zeta d\eta.
\label{Hest}
\end{eqnarray}
Here we used the change of variables $\xi\mapsto\zeta$ via 
$$
\zeta=\tau_1(t,x,\xi,\la)=t-\int_x^\xi\frac{d\eta}{a(\eta,\la)},
$$
and introduced the function $\om(t,x,\cdot,\la)$ as the inverse to $\tau_1(t,x,\cdot,\la)$. From \reff{Hest} it follows
$$
|(\HH(\la)v)_1(t,x)|+|\d_t(\HH(\la)v)_1(t,x)|+|\d_x(\HH(\la)v)_1(t,x)| \le \mbox{const} \|v\|_0,
$$
where the constant does not depend on $t,x$, $\la \approx 0$ and $v$.
This implies \reff{Hest1} as desired.
\end{proof}

Now from  Lemma~\ref{B2} it follows
\begin{cor} 
\label{corro1a}
For all $\la \in [-\eps_1,\eps_1]$ the following is true:

(i) The operator $I-\CC(\la)-\D(\la)$ is Fredholm of index zero from $U_0$ into $U_0$.

(ii) $\ker (I-\CC(\la)-\D(\la)) \subset V$.
\end{cor}

\begin{lem}
\label{B5anwa}
The operator $I-\CC(0)-\D(0)$ is injective.
\end{lem}
\begin{proof}
Suppose $(I-\CC(0)-\D(0))v=0$ for some $v \in U_0$. Then $v \in V$ by Corollary~\ref{corro1a}~(ii).
Hence, $u=J_0^\la v$ is a solution of  the linearized problem
\reff{evpa}. But by the assumption of Theorem \ref{thm:hopfa}, the problem \reff{evpa} does not have nontrivial solutions.
\end{proof}

\begin{cor}
\label{corroa}
For all positive integers $l$ it holds $v_0 \in U_l$.
\end{cor}
\begin{proof}
Because of Lemma \ref{propa}  the map $F_0:\R \times U_0 \to U_0$, which is defined by 
$$
F_0(s,v):=T(s)(I-C(0,T(-s)v)T(-s)v-D(0,T(-s)v,T(-s)v)),
$$
is $C^\infty$-smooth. Moreover, it holds $F_0(s,T(s)v_0)=0$ for all $s \in \R$, in particular,  $F_0(0,v_0)=0$. Finally, from
Lemma  \ref{B5anwa}  it follows that $\d_uF_0(0,v_0)$ is an isomorphism from $U_0$ onto $U_0$. Hence, 
the classical implicit function theorem yields that the map $s \approx 0
\mapsto T(s)v_0 \in U_0$ is  $C^\infty$-smooth, i.e., $v_0 \in U_l$ for all $l$.
\end{proof}

Now we represent the problem  (\ref{eq:1.1a})--(\ref{eq:1.2a}), i.e., the  problem  (\ref{Gl})--(\ref{bcond}), 
in the setting of Section~\ref{IFT}. Define
$$
F: [-\eps_0,\eps_0] \times U_0\to U_0:\; F(\la,v):=v-C(\la,v_0)v-D(\la,v,v_0).
$$
Then, because of Lemma \ref{propa}  and Corollary \ref{corroa}, the map
$$
(s,v) \in \R \times U_0 \mapsto \F(\la,s,v):=T(s)F(\la,T(-s)v) \in U_0
$$
is $C^\infty$-smooth for all $\la \in  [-\eps_0,\eps_0]$, i.e., the condition \reff{infsmooth} is fulfilled.
Moreover, it is easy to show that the conditions \reff{ksmooth} and \reff{apriori} are fulfilled.
Further, Lemma \ref{corro1a} (i) yields  \reff{FH}. And  finally, Lemmas \ref{B4anwa}, 
 \ref{B5anwa} and 
  \ref{B2} and  Remark \ref{B3} (with $U:=U_0$) imply 
\reff{coerz}. Hence, applying Theorem  \ref{thm:IFT}  to the equation $v=C(\la,v_0)v+D(\la,v,v_0)$, 
we finish the proof of  Theorem \ref{thm:hopfa}.

\begin{appendix}

\section{Appendix
}

\label{AppendixA}

\renewcommand{\theequation}{A.\arabic{equation}}
\setcounter{equation}{0}

\setcounter{thm}{0}
    \renewcommand{\thethm}{A.\arabic{thm}}

In this appendix we present a simple linear version of the so-called fiber contraction principle (see, e.g., \cite[Section 1.11.3]{Chicone}):

\begin{lem}
\label{A2}
Let $U$ be a Banach space and $u_1,u_2,\ldots \in U$ a converging sequence. Further, let $A_1,A_2,\ldots \in \LL(U)$
be a sequence of linear bounded operators on $U$ such that there exists $c<1$ such that for all $u \in U$ it holds
\beq
\label{a1}
\|A_n u\| \le c \|u\| \mbox{ for all } n=1,2,\ldots
\ee
and \beq
\label{a2}
A_1u , A_2u,\ldots \,\,\mbox{\rm converges in  }  U.
\ee
Finally, let $v_1,v_2,\ldots \in U$ be a sequence such that for all positive integers $n$ we have 
\beq
\label{a3}
v_{n+1}=A_nv_n+u_n.
\ee
Then the sequence  $v_1,v_2,\ldots$ converges in $U$.
\end{lem}
\begin{proof}
Because of \reff{a2} there exists $A \in \LL(U)$ such that for all $u \in U$ we have $A_nu \to Au$ in $U$ for $n \to \infty$.
Moreover, \reff{a1} yields that $\|Au\| \le c\|u\|$ for all $u \in U$. Because of $c < 1$ there exists exactly one $v \in U$ such that
$v=Av+u$, where $u \in U$ is the limit of the sequence $u_1,u_2,\ldots \in U$.

Let us show that $v_k \to v$ in $U$  for $k \to \infty$.
Because of \reff{a3} we have 
\beq
\label{a}
v_{n+1}-v
=A_n(v_{n}-v)+\left(A_n-A\right)v+u_{n}-u.
\ee
Using the notation
$r_{n}:=\|v_n-v\|$ and 
$s_{n}:=\|\left(A_{n}-A\right)v\| + \|u_{n}-u\|$,
we get from \reff{a1} and \reff{a}
\beq
r_{n+1} \le c r_{n} +s_{n} \le c^2  r_{n-1}+cs_{n-1}+s_{n}
\le \ldots\le c^{n}r_{1}+\sum_{m=0}^n c^{n-m} s_{m}.
\label{riteration}
\ee
Take $s_0>0$ such 
that $s_{n} \le s_0$ for all $n$.
Then \reff{riteration} implies for $0 \le l \le n$
\beq
r_{n+1} 
\le c^n \left(r_1+s_0\sum_{m=0}^l c^{-m}\right)+\frac{1}{1-c}\max_{l<m\le n}s_{m}.
\label{riteration1}
\ee
But we have
\beq
\label{max}
\max_{l<m\le n}s_{m} \to 0 \mbox{ for } l \to \infty \mbox{ uniformly in } n,
\ee
and the assumption $c<1$ implies
\beq
\label{sum}
 c^n \left(r_1+s_0\sum_{m=0}^l c^{-m}\right) \to 0 \mbox{ for } n \to \infty \mbox{ for any fixed } l.
\ee
Now \reff{riteration1}--\reff{sum} yield that $r_{n} \to 0$ for $n \to \infty$,
which is the claim.
\end{proof}

\end{appendix}

\begin{appendix}
\section{
Appendix 
}

\label{AppendixB}
\renewcommand{\theequation}{B.\arabic{equation}}
\setcounter{equation}{0}

\setcounter{thm}{0}
    \renewcommand{\thethm}{B.\arabic{thm}}

In this appendix we present an approach how to verify the assumptions \reff{FH} and \reff{coerz} of Theorem \ref{thm:IFT}. For similar approaches see 
\cite{Magnus,RO}.

\begin{lem}
\label{B2}
Let $(U,\|\cdot\|_U)$ and  $(V,\|\cdot\|_V)$ be Banach spaces and  $(H,\langle\cdot,\cdot\rangle)$ be a Hilbert space such that
\beq
\label{VU}
\mbox{$V$ is compactly embedded into $U$}
\ee
and
\beq
\label{UH}
\mbox{$U$ is continuously embedded into $H$.}
\ee
Further, let $\eps_0>0$, and for any $\la \in [-\eps_0,\eps_0]$ let  $\CC(\la),\D(\la) \in {\cal L}(U)$ 
be such that
\beq
\label{Hlim}
\lim_{\la \to 0} \sup_{\|u\|_U\le 1}\langle \left(\CC(\la)+\D(\la)-\CC(0)-\D(0)\right)u,h\rangle=0 \mbox{ for all } h \in H.
\ee
Finally, suppose that there exists $d>0$ such that for all  $\la \in [-\eps_0,\eps_0]$
\beq
\label{1}
I-\CC(\la) \mbox{ is bijective from } U \mbox{ to } U 
\mbox{ and } \left\|\left(I-\CC(\la)\right)^{-1}\right\|_{\LL(U)} \le d,
\ee
\beq
\label{CDabsch}
\|\CC(\la)\|_{\LL(U)}+\|\D(\la)\|_{\LL(U)} \le d,
\ee
and 
\beq
\label{absch}
\left(\left(I-\CC(\la)\right)^{-1}\D(\la)\right)^2u \in V, \;\left\|\left(\left(I-\CC(\la)\right)^{-1}\D(\la)\right)^2u\right\|_V 
\le d\|u\|_U \mbox{ for all } u \in U.
\ee 
Then the following is true:

(i) $I-\CC(\la)-\D(\la)$ is Fredholm of index zero from $U$ to $U$ for all $ \la \in [-\eps_0,\eps_0]$.

(ii) If  $\left(I-\CC(\la)-\D(\la)\right)u=0$ for $u \in U$ and  $\la \in [-\eps_0,\eps_0]$, then $u \in V$.

(iii) If $I-\CC(0)-\D(0)$ is injective, then
there exist $\eps_1 \in (0,\eps_0]$ and $c>0$ such that $\|\left(I-\CC(\la)-\D(\la)\right)u\|_U \ge c\|u\|_U$  for all $ \la \in [-\eps_1,\eps_1]$ and $u \in U$.
\end{lem}
\begin{proof} 
Because of assumptions \reff{1} and \reff{absch} the operator   $\left((I-\CC(\la))^{-1}\D(\la)\right)^2$ is bounded from $U$ to $V$. Hence, assumption  \reff{VU}
yields that it is compact from $U$ to $U$. Therefore  \cite[Theorem XIII.5.2]{KA} implies that the operator  $I-(I-\CC(\la))^{-1}\D(\la)$ is Fredholm of index zero from $U$ to $U$,
i.e. Claim (i) is therewith proved.

To prove Claim (ii), take $u \in U$ with $\left(I-\CC(\la)-\D(\la)\right)u=0$. Then  we have 
$u=\left((I-\CC(\la))^{-1}\D(\la)\right)^2u$,
and  \reff{absch} yields $u \in V$, as desired.

To prove Claim (iii), suppose the contrary. Then there exist sequences $\la_1,\la_2,\ldots \in  [-\eps_0,\eps_0]$ and $u_1,u_2,\ldots \in U$ with 
\beq
\label{un1}
\|u_n\|_U=1
\ee
and
\beq
\label{Aun1}
|\la_n|+\|\left(I-\CC(\la_n)-\D(\la_n)\right)u_n\|_U \to 0.
\ee
From \reff{1}, \reff{CDabsch} and \reff{Aun1} it follows
\begin{eqnarray}
\label{Aun2}
&&\left\|\left(I-\left((I-\CC(\la_n))^{-1}\D(\la_n)\right)^2\right)u_n\right\|_U\nonumber\\
&&=\left\|(I-\CC(\la_n))^{-1}(I-\CC(\la_n)+\D(\la_n))(I-\CC(\la_n))^{-1}\left(I-\CC(\la_n)-\D(\la_n)\right)u_n\right\|_U \nonumber\\
&&\le \mbox{ const } \|\left(I-\CC(\la_n)-\D(\la_n)\right)u_n\|_U \to 0.
\end{eqnarray}
Moreover, because of  \reff{absch} and\reff{un1}
the sequence
$
\left((I-\CC(\la_n))^{-1}\D(\la_n)\right)^2u_n
$
is bounded in $V$.
Hence, \reff{VU} yields that without loss of generality we can assume that it converges in $U$, 
i.e., that there exists $u_* \in U$ such that
$$
\left\|\left(I-\CC(\la_n))^{-1}\D(\la_n)\right)^2u_n-u_*\right\|_U \to 0.
$$
Therefore,
\reff{Aun2} implies
\beq
\label{Aun3}
\|u_n-u_*\|_U \to 0.
\ee
On the other hand, from  \reff{Hlim} and  \reff{un1} we have 
$$
\left(\CC(\la_n)+\D(\la_n)-\CC(0)-\D(0)\right)u_n \rightharpoonup 0 \mbox{ in } H.
$$
Using  \reff{UH} and  \reff{Aun1}, we get $(I-\CC(0)-\D(0))u_n \rightharpoonup 0$ in $H$. Hence,   \reff{UH} and  \reff{Aun3} yield $(I-\CC(0)-\D(0))u_*=0$. 
But the operator $I-\CC(0)-\D(0)$ is injective, hence $u_*=0$.
Therefore \reff{Aun3} contradicts to  \reff{un1}.
\end{proof}

\begin{rem}\rm
\label{B3}
If $V$ is dense in $U$, then for \reff{absch} it is sufficient that there exists $d>0$ such that for all $\la \in  [-\eps_0,\eps_0]$ 
\beq
\label{Vabsch}
\D(\la)^2v, \D(\la)\CC(\la)v \in V \mbox{ and } \|\D(\la)^2v\|_V + \|\D(\la)\CC(\la)v\|_V 
\le d\|v\|_U \mbox{ for all } v \in V.
\ee 
Indeed, take $u \in U$. Then there exists a sequence $v_1,v_2,\ldots \in V$ such that $\|v_n-u\|_U \to 0$. Because of \reff{1},  \reff{CDabsch} and \reff{Vabsch},
the sequence
\beq
\label{formel}
\left((I-\CC(\la))^{-1}\D(\la)\right)^2v_n=(I-\CC(\la))^{-1}\left(\D(\la)^2+\D(\la)\CC(\la)(I-\CC(\la))^{-1}\D(\la)\right)v_n
\ee
is fundamental in $V$. Hence, there exists $v_* \in V$ such that 
$$
\left\|\left((I-\CC(\la))^{-1}\D(\la)\right)^2v_n-v_*\right\|_V \to 0.
$$ 
On the other hand, the sequence 
$\left((I-\CC(\la))^{-1}\D(\la)\right)^2v_n$ tends to $\left((I-\CC(\la))^{-1}\D(\la)\right)^2u$ in $U$, i.e. $\left((I-\CC(\la))^{-1}\D(\la)\right)^2u=v_*$.
Using    \reff{1}, \reff{CDabsch}, \reff{Vabsch} and \reff{formel} again we get
\begin{eqnarray*}
\|\left((I-\CC(\la))^{-1}\D(\la)\right)^2u\|_V&=&\lim_{n \to \infty}\|\left((I-\CC(\la))^{-1}\D(\la)\right)^2v_n\|_V\\
&\le & \mbox{\rm const }\lim_{n \to \infty}\|v_n\|_U
=  \mbox{\rm const }\|u\|_U.
\end{eqnarray*}
\end{rem}

\end{appendix}

\end{document}